\documentclass[12pt]{amsart}
\usepackage{amssymb}

\let\Bbb\mathbb

\def\>{\relax\ifmmode\mskip.666667\thinmuskip\relax\else\kern.111111em\fi}
\def\<{\relax\ifmmode\mskip-.333333\thinmuskip\relax\else\kern-.0555556em\fi}
\def\vsk#1>{\vskip#1\baselineskip}
\def\vv#1>{\vadjust{\vsk#1>}\ignorespaces}
\def\vvn#1>{\vadjust{\nobreak\vsk#1>\nobreak}\ignorespaces}
\def\vvgood{\vadjust{\penalty-500}}
\let\alb\allowbreak

\def\sskip{\par\vsk.2>}
\let\Medskip\medskip
\def\medskip{\par\Medskip}
\let\Bigskip\bigskip
\def\bigskip{\par\Bigskip}

\let\Maketitle\maketitle
\def\maketitle{\Maketitle\thispagestyle{empty}\let\maketitle\empty}

\newtheorem{thm}{Theorem}[section]
\newtheorem{cor}[thm]{Corollary}
\newtheorem{lem}[thm]{Lemma}
\newtheorem{prop}[thm]{Proposition}

\numberwithin{equation}{section}

\theoremstyle{definition}
\newtheorem*{rem}{Remark}
\newtheorem*{example}{Example}

\let\mc\mathcal
\let\nc\newcommand

\nc{\on}{\operatorname}
\nc{\Z}{{\mathbb Z}}
\nc{\C}{{\mathbb C}}
\nc{\N}{{\mathbb N}}
\nc{\pone}{{\mathbb C}{\mathbb P}^1}
\nc{\arr}{\rightarrow}
\nc{\larr}{\longrightarrow}
\nc{\al}{\alpha}
\nc{\W}{{\mc W}}
\nc{\la}{\lambda}
\nc{\su}{\widehat{{\mathfrak sl}}_2}
\nc{\g}{{\mathfrak g}}
\nc{\h}{{\mathfrak h}}
\nc{\m}{{\mathfrak m}}
\nc{\n}{{\mathfrak n}}
\nc{\Gm}{\Gamma}
\nc{\La}{\Lambda}
\nc{\gl}{\widehat{\mathfrak{gl}_2}}
\nc{\bi}{\bibitem}
\nc{\om}{\omega}
\nc{\Res}{\on{Res}}
\nc{\gm}{\gamma}
\nc{\Om}{\Omega}

\def\fratop{\genfrac{}{}{0pt}1}

\def\Mu{\mathrm M}
\def\z{\mathfrak z}
\def\ev{\mbox{\sl ev}}
\def\ch{\on{ch}}
\def\gr{\on{gr}}
\def\Ann{\on{Ann}}
\def\End{\on{End}}
\def\Gr{\on{Gr}}
\def\Res{\on{Res}}
\def\rdet{\on{rdet}}
\def\Wr{\on{Wr}}

\def\tbigcap{\mathrel{\textstyle{\bigcap}}}
\def\lab{{\bs{\bar\la}}}
\def\mut{\bs{\tilde\mu}}
\def\lba{{\bs\la\>,\bs a}}

\def\Llb{{\bs\La,\bs\la\>,\bs b}}
\def\Lbl{{\bs\La,\lab\>,\bs b}}
\def\Mmb{{\bs\Mu,\bs b}}
\def\Wrl{\Wr_{\bs\la}}
\def\Wrli{\Wrl^{-1}}
\def\Wrnb{\Wr_{\bs n,\bs b}}
\def\Wrnbi{\Wrnb^{-1}}
\def\Dlnb{\Dl_{\bs n,\bs b}}

\def\B{{\mc B}}
\def\D{{\mc D}}

\def\F{{\mc F}}
\def\L{{\mc L}}
\def\M{{\mc M}}
\def\O{{\mc O}}
\def\Q{{\mc Q}}
\def\V{{\mc V}}
\def\Dt{\Tilde\D}

\def\flati{\def\=##1{\rlap##1\hphantom b)}}

\let\dl\delta
\let\Dl\Delta
\let\eps\varepsilon
\let\si\sigma

\let\Sig\varSigma
\let\Tht\Theta

\let\Tilde\widetilde
\let\der\partial

\let\ge\geqslant
\let\geq\geqslant
\let\le\leqslant
\let\leq\leqslant

\let\bat\bar

\nc{\gln}{\mathfrak{gl}_N}
\nc{\sln}{\mathfrak{sl}_N}

\def\glnt{\gln[t]}
\def\Ugln{U(\gln)}
\def\Uglnt{U(\glnt)}

\def\slnt{\sln[t]}

\def\beq{\begin{equation}}
\def\eeq{\end{equation}}
\def\be{\begin{equation*}}
\def\ee{\end{equation*}}
\nc{\bean}{\begin{eqnarray}}
\nc{\eean}{\end{eqnarray}}
\nc{\bea}{\begin{eqnarray*}}
\nc{\eea}{\end{eqnarray*}}
\nc{\bs}{\boldsymbol}
\nc{\Ref}[1]{{\rm(\ref{#1})}}

\nc{\R}{\Bbb R}

\begin{document}

\hrule width0pt
\vsk->

\title[Schubert calculus and representations of general linear group]
{Schubert calculus and representations of\\[4pt] general linear group}

\author[E.\,Mukhin, V.\,Tarasov, and A.\,Varchenko]
{E.\,Mukhin$\>^*$, V.\,Tarasov$\>^\star$, and A.\,Varchenko$\>^\diamond$}

\maketitle

\begin{center}
{\it $\kern-.4em^{*,\star}\<$Department of Mathematical Sciences,
Indiana University\,--\>Purdue University Indianapolis\kern-.4em\\
402 North Blackford St, Indianapolis, IN 46202-3216, USA\/}

\medskip
{\it $^\star\<$St.\,Petersburg Branch of Steklov Mathematical Institute\\
Fontanka 27, St.\,Petersburg, 191023, Russia\/}

\medskip
{\it $^\diamond\<$Department of Mathematics, University of North Carolina
at Chapel Hill\\ Chapel Hill, NC 27599-3250, USA\/}
\end{center}

{\let\thefootnote\relax
\footnotetext{\vsk-.8>\noindent
$^*\<${\sl E\>-mail}:\enspace mukhin@math.iupui.edu\\
$^\star\<${\sl E\>-mail}:\enspace vt@math.iupui.edu\>, vt@pdmi.ras.ru\\
$^\diamond\<${\sl E\>-mail}:\enspace anv@email.unc.edu}}

\medskip
\begin{abstract}
We construct a canonical isomorphism between the Bethe algebra
acting on a multiplicity space of a tensor product of evaluation
$\gln[t]$-modules and the scheme-theoretic intersection of suitable
Schubert varieties. Moreover, we prove that the multiplicity space
as a module over the Bethe algebra is isomorphic to the coregular
representation of the scheme-theoretic intersection.

In particular, this result implies the simplicity of the spectrum of
the Bethe algebra for real values of evaluation parameters and the
transversality of the intersection of the corresponding Schubert
varieties.
\end{abstract}

\setcounter{footnote}{0}
\renewcommand{\thefootnote}{\arabic{footnote}}

\section{Introduction}

It is known for a long time that the intersection index of the Schubert
varieties in the Grassmannian of the $N$-dimensional planes coincides with
the dimension of the space of invariant vectors in a suitable tensor product
of finite-dimensional irreducible representations of the general linear group
$\gln$.

There is a natural commutative algebra acting on such a space of
invariant vectors, which we call the Bethe algebra. The Bethe algebra is 
the central object of study in the quantum Gaudin model, 
see \cite{G}. 
In this paper,
we construct an algebra isomorphism between the Bethe algebra and
the scheme-theoretic intersection of the Schubert varieties. We also show that
the representation of the Bethe algebra on the space of the invariant vectors
is naturally isomorphic to the coregular representation of the scheme-theoretic
intersection of the Schubert varieties (in fact, the regular and coregular
representations of that algebra are isomorphic). Thus we explain and generalize
the above mentioned coincidence of the numbers.

The existence of such a strong connection between seemingly unrelated subjects:
Schubert calculus and integrable systems, turns out to be very advantageous for
both.

On the side of Schubert calculus, using a weaker form of the relation
we succeeded to prove the B.~and M.\,Shapiro conjecture in~\cite{MTV2}.
\vvgood
An alternative proof is given in the present paper. Moreover, we are now able
to deduce that the Schubert varieties corresponding to real data intersect
transversally, see Corollary~\ref{real}. The transversality of the intersection
was a long standing conjecture, see~\cite{EH}, \cite{S}.

\sskip
The transversality property can be reformulated as the following statement:
the number of monodromy free monic Fuchsian differential
operators of order $N$ with $k+1$ singular points, all of them lying on a line
or a circle, and prescribed exponents at the singular points equals the
multiplicity of the trivial $\gln$-module in the tensor product of $k+1$
irreducible $\gln$-modules with highest weights determined by the exponents at
the critical points.

\sskip
On the side of the integrable systems, we obtain a lot
of new information about
the Bethe algebra, see Theorem~\ref{BL}. In particular, we are able to show 
that all
eigenspaces of the Bethe algebra are one-dimensional for all
values of parameters. Since it is known that the Bethe algebra is
diagonalizable for real data, our result implies that the spectrum of the Bethe
algebra is simple in this case and hence it is simple generically. We also show
that the Bethe algebra is a maximal commutative subalgebra in the algebra of
linear operators.

In addition, an immediate corollary of our result is a bijective correspondence
between eigenvalues of the Bethe algebra and monic differential operators whose
kernels consist of polynomials only, see Theorem~\ref{BL}, parts~(v), (vi).
The obtained correspondence between the spectrum of the Bethe algebra and
the differential operators is in the spirit of the geometric Langlands
correspondence, see~\cite{F}.

\sskip
We obtained similar results for the Lie algebra $\mathfrak{gl}_2$ and the
Grassmannian of two-dimen\-sional planes in~\cite{MTV3}. The crucial difference
between~\cite{MTV3} and the present paper is that we are able to avoid using
the so-called weight function. The weight function in the case of
$\mathfrak{gl}_2$ can be handled with the help of the functional Bethe ansatz.
In the case of $\gln$, $N>2$, the weight function is a much more complicated
object and the functional Bethe ansatz is not yet sufficiently developed.
We hope that the results of this paper 
can be used to regularize the weight function.

\sskip
In the present paper we use the following approach. Let $V=\C^N$ be
the vector representation of $\gln$. We consider the $\gln[t]$-module
$\V^S=(V^{\otimes n}\otimes \C[z_1,\dots,z_n])^S$ which is a
cyclic submodule of the tensor product of evaluation modules with
formal evaluation parameters. We show that the Bethe algebra
$\B_{\bs\la}$ acting on the subspace $(\V^S)_{\bs\la}^{sing}$ of
singular vectors of weight $\bs\la$ is a free polynomial algebra,
see Theorem~\ref{first} and its representation on $(\V^S)_{\bs\la}^{sing}$ is
isomorphic to the regular representation, see Theorem~\ref{first1}. To get this
result we exploit three observations: a natural identification of the algebra
$\B_{\bs\la}$ with the algebra of functions on a suitable Schubert cell, the
completeness of the Bethe Ansatz for the tensor product of the vector
representations at generic evaluation points, and the equality of the graded
characters of the Bethe algebra $\B_{\bs\la}$ and of the space
$(\V^S)_{\bs\la}^{sing}$.

At the second step we specialize the evaluation parameters, and
consider a quotient of the $\gln[t]$-module $\V^S$ which is
isomorphic to the Weyl module. We take the corresponding quotient of
the algebra of functions on the Schubert cell and obtain an
isomorphism between the scheme-theoretic fiber of the Wronski map and
the Bethe algebra $\B_\lba$ acting on the subspace of
singular vectors of weight $\bs\la$ in the Weyl module.
We also show that this representation of the algebra
$\B_\lba$ is isomorphic to the regular representation,
see Theorem~\ref{second}. In fact, the algebra $\B_{\bs\la,\bs a}$
is Frobenius, so its regular and coregular representations are isomorphic.

\newpage
Finally, we impose more constraints to descend to the scheme-theoretic
intersection of the Schubert varieties on one side and to the Bethe
algebra $\B_{\bs\La,\bs\la,\bs b}$ acting on a tensor product of
evaluation modules with highest weights
$\bs\la^{(1)},\dots,\bs\la^{(k)}$ and evaluation parameters
$b_1,\dots,b_k$ on the other side, see Theorem~\ref{third}.

The paper is organized as follows. In Section~\ref{alg sec}, we discuss
representations of the current algebra $\glnt$ and
introduce the Bethe algebra $\B$ as a subalgebra of $U(\glnt)$.
In Section~\ref{sch sec}, we study the algebra of
functions of a Schubert cell and the associated Wronski map.
The scheme-theoretic intersection of Schubert cells is considered
in Section~\ref{schubert}. The central result there is Proposition~\ref{OIX},
which describes this algebra by generators and relations.
We prove the main results of the paper, Theorems~\ref{first}, \ref{first1},
\ref{second}, and ~\ref{third} in Section~\ref{iso sec}.
Section~\ref{app sec} describes applications.

\subsection*{Acknowledgments}
We are grateful to P.\,Belkale, L.\,Borisov, V.\,Chari, P.\,Etingof, B.
Feigin, A.\,Kirillov, M.\,Nazarov and F.\,Sottile
for valuable discussions.

E.\,Mukhin is supported in part by NSF grant DMS-0601005.
V.\,Tarasov is supported in part by RFFI grant 05-01-00922.
A.\,Varchenko is supported in part by NSF grant DMS-0555327.
\vvn-.6>

\section{Representations of current algebra $\glnt$}
\label{alg sec}
\subsection{Lie algebra $\gln$}
Let $e_{ij}$, $i,j=1,\dots,N$, be the standard generators of the complex
Lie algebra
$\gln$ satisfying the relations
$[e_{ij},e_{sk}]=\dl_{js}e_{ik}-\dl_{ik}e_{sj}$. We identify the Lie algebra
$\sln$ with the subalgebra in $\gln$ generated by the elements
$e_{ii}-e_{jj}$ and $e_{ij}$ for $i\ne j$, $i,j=1,\dots,N$.

The subalgebra $\z_N\subset\gln$ generated by the element
$\sum_{i=1}^Ne_{ii}\,$ is central. The Lie algebra $\gln$ is canonically
isomorphic to the direct sum $\sln\oplus\z_N$.

Given an $N\times N$ matrix $A$ with possibly noncommuting entries $a_{ij}$,
we define its {\it row determinant\/} to be
\beq
\rdet A\,=
\sum_{\;\si\in S_N\!} (-1)^\si\,a_{1\si(1)}a_{2\si(2)}\dots a_{N\si(N)}\,.
\eeq

Let $Z(x)$ be the following polynomial in variable $x$ with coefficients
in $\Ugln$:
\vvn.3>
\beq
\label{Zx}
Z(x)\,=\,\rdet\left( \begin{matrix}
x-e_{11} & -\>e_{21}& \dots & -\>e_{N1}\\
-\>e_{12} &x+1-e_{22}& \dots & -\>e_{N2}\\
\dots & \dots &\dots &\dots \\
-\>e_{1N} & -\>e_{2N}& \dots & x+N-1-e_{NN}
\end{matrix}\right).
\vv.3>
\eeq

The following statement is proved in~\cite{HU},
see also Section~2.11 in~\cite{MNO}.
\begin{thm}\label{Zcent}
The coefficients of the polynomial\/ $\,Z(x)-x^N$ are free generators
of the center of\/ $\Ugln$.
\qed
\end{thm}

Let $M$ be a $\gln$-module. A vector $v\in M$ has weight
$\bs\la=(\la_1,\dots,\la_N)\in\C^N$ if $e_{ii}v=\la_iv$ for $i=1,\dots,N$.
A vector $v$ is called {\it singular\/} if $e_{ij}v=0$ for $1\le i<j\le N$.
If $v$ is a singular of weight $\bs\la$, then clearly
\vvn-.5>
\beq
\label{Zxv}
Z(x)\,v\,=\,\prod_{i=1}^N\,(x-\la_i+i-1)\cdot v\,.
\eeq

We denote by $(M)_{\bs\la}$ the subspace of $M$ of weight $\bs\la$,
by $M^{sing}$ the subspace of $M$ of all singular vectors and by
$(M)^{sing}_{\bs\la}$ the subspace of $M$ of all singular vectors
of weight $\bs\la$.

Denote by $L_{\bs\la}$ the irreducible finite-dimensional $\gln$-module with
highest weight $\bs\la$. Any finite-dimensional $\gln$-module $M$ is isomorphic
to the direct sum $\bigoplus_{\bs\la}L_{\bs\la}\otimes(M)_{\bs\la}^{sing}$,
where the spaces $(M)_{\bs\la}^{sing}$ are considered as trivial
$\gln$-modules.

The $\gln$-module $L_{(1,0,\dots,0)}$ is the standard $N$-dimensional vector
representation of $\gln$. We denote it by $V$. We choose a highest weight
vector in $V$ and denote it by $v_+$.

A $\gln$-module $M$ is called polynomial if it is isomorphic to a submodule of
$V^{\otimes n}$ for some $n$.

A sequence of integers $\bs\la=(\la_1,\dots,\la_N)$ such that
$\la_1\ge\la_2\ge\dots\ge\la_N\ge0$ is called a {\it partition with $N$
parts\/}. Set $|\bs\la|=\sum_{i=1}^N\la_i$. Then it is said that $\bs\la$
is a partition of $|\bs\la|$.

\sskip
The $\gln$-module $V^{\otimes n}$ contains the module $L_{\bs\la}$
if and only if $\bs\la$ is a partition of $n$ with at most $N$ parts.

\sskip
For a Lie algebra $\g\,$, we denote by $U(\g)$ the universal enveloping algebra
of $\g$.

\subsection{Current algebra $\glnt$}
Let $\glnt=\gln\otimes\C[t]$ be the complex
Lie algebra of $\gln$-valued polynomials
with the pointwise commutator. We call it the {\it current algebra\/}.
We identify the Lie algebra $\gln$ with the subalgebra $\gln\otimes1$
of constant polynomials in $\glnt$. Hence, any $\glnt$-module has the canonical
structure of a $\gln$-module.

The standard generators of $\glnt$ are $e_{ij}\otimes t^r$, $i,j=1,\dots,N$,
$r\in\Z_{\ge0}$. They satisfy the relations
$[e_{ij}\otimes t^r,e_{sk}\otimes t^p]=
\dl_{js}e_{ik}\otimes t^{r+p}-\dl_{ik}e_{sj}\otimes t^{r+p}$.

The subalgebra $\z_N[t]\subset\glnt$ generated by the elements
$\sum_{i=1}^Ne_{ii}\otimes t^r$, $r\in\Z_{\ge0}$, is central.
The Lie algebra $\glnt$ is canonically isomorphic to the direct sum
$\slnt\oplus\z_N[t]$.

It is convenient to collect elements of $\glnt$ in generating series
of a formal variable $u$. For $g\in\gln$, set
\vvn-.2>
\be
g(u)=\sum_{s=0}^\infty (g\otimes t^s)u^{-s-1}.
\vv.2>
\ee
We have $(u-v)[e_{ij}(u),e_{sk}(v)] =
\dl_{js}(e_{ik}(u)-e_{ik}(v)) -
\dl_{ik}(e_{sj}(u)-e_{sj}(v))$.

For each $a\in\C$, there exists an automorphism $\rho_a$ of $\glnt$,
\;$\rho_a:g(u)\mapsto g(u-a)$. Given a $\glnt$-module $M$, we denote by $M(a)$
the pull-back of $M$ through the automorphism $\rho_a$. As $\gln$-modules,
$M$ and $M(a)$ are isomorphic by the identity map.

For any $\glnt$-modules $L,M$ and any $a\in\C$, the identity map
$(L\otimes M)(a)\to L(a)\otimes M(a)$ is an isomorphism of $\glnt$-modules.

We have the evaluation homomorphism,
${\ev:\glnt\to\gln}$, \;${\ev:g(u) \mapsto g\>u^{-1}}$.
Its restriction to the subalgebra $\gln\subset\glnt$ is the identity map.
For any $\gln$-module $M$, we denote by the same letter the $\glnt$-module,
obtained by pulling $M$ back through the evaluation homomorphism. For each
$a\in\C$, the $\glnt$-module $M(a)$ is called an {\it evaluation module\/}.

If $b_1,\dots,b_n$ are distinct complex numbers and $L_1,\dots,L_n$ are
irreducible finite-dimensional $\gln$-modules, then the $\glnt$-module
$\otimes_{s=1}^n L_s(b_s)$ is irreducible.

We have a natural $\Z_{\ge0}$-grading on $\glnt$ such that for any $g\in\gln$,
the degree of $g\otimes t^r$ equals $r$. We set the degree of $u$ to be $1$.
Then the series $g(u)$ is homogeneous of degree $-1$.

A $\glnt$-module is called {\it graded\/} if it
has a $\Z_{\ge0}$-grading compatible with the grading of $\glnt$.
Any irreducible graded $\glnt$-module is isomorphic to an evaluation module
$L(0)$ for some irreducible $\gln$-module $L$, see~\cite{CG}.

Let $M$ be a $\Z_{\ge0}$-graded space with finite-dimensional homogeneous
components. Let $M_j\subset M$ be the homogeneous component of degree $j$.
We call the formal power series in variable $q$,
\vvn-.2>
\be
\ch(M)=\sum_{j=0}^\infty\,(\dim M_j)\,q^j\,,
\vv-.2>
\ee
the {\it graded character\/} of $M$.

\subsection{Weyl modules}
\label{secweyl}
Let $W_m$ be the $\glnt$-module generated by a vector $v_m$ with the
defining relations:
\vvn-.2>
\begin{alignat*}2
e_{ii} &(u)v_m=\>\dl_{1i}\,\frac mu\,v_m\,, && i=1,\dots,N\,,
\\[4pt]
e_{ij} & (u)v_m=\>0\,, && 1\le i<j\le N\,,
\\[4pt]
(e_{ji} &{}\otimes1)^{m\dl_{1i}+1}v_m=\>0\,,\qquad && 1\le i<j\le N\,.
\\[-12pt]
\end{alignat*}
As an $\slnt$-module, the module $W_m$ is isomorphic to the Weyl module from
\cite{CL}, \cite{CP}, corresponding to the weight $m\om_1$, where $\om_1$ is
the first fundamental weight of $\sln$. Note that $W_1=V(0)$.

\begin{lem}
\label{weyl}
The module $W_m$ has the following properties.
\begin{enumerate}
\item[(i)] The module \>$W_m$ has a unique grading such that\/ \>$W_m$ is
a graded $\glnt$-module and such that
the degree of\/ $v_m$ equals\/ $0$.
\item[(ii)] As a $\gln$-module, $W_m$ is isomorphic to $V^{\otimes m}$.
\item[(iii)] A $\glnt$-module $M$ is an irreducible subquotient of\/ $W_m$
if and only if\/ $M$ has the form $L_{\bs\la}(0)$,
where\/ $\bs\la$ is a partition of\/ $m$ with at most $N$ parts.
\item[(iv)] For any partition $\bs\la$ of\/ $m$ with at most $N$ parts,
the graded character of the space $(W_m)_{\bs\la}^{sing}$ is given by
\vvn-.2>
\be
\ch((W_m)_{\bs\la}^{sing})\,=\,
\frac{(q)_m\prod_{1\le i<j\le N}(1-q^{\la_i-\la_j+j-i})}
{\prod_{i=1}^N(q)_{\la_i+N-i}}\ q^{\sum_{i=1}^N{(i-1)\la_i }},
\ee
where\/ $\,(q)_a=\prod_{j=1}^a(1-q^j)\,$.
\end{enumerate}
\end{lem}
\begin{proof}
The first two properties are proved in~\cite{CP}. The third property follows
from the first two. The last property is well-known, see for example the
relation between the character $\ch\bigl((W_n)_{\bs\la}^{sing}\bigr)$ and
the Kostka polynomials in~\cite[Corollary 1.5.2]{CL}, and use the formula
for the corresponding Kostka polynomial in~\cite[Example~2, p.\,243]{M}.
\end{proof}

For each $b\in\C$, the $\glnt$-module $W_m(b)$ is cyclic with a cyclic vector
$v_m$.

\begin{lem}
\label{weylb}
The module $W_m(b)$ has the following properties.
\begin{enumerate}
\item[(i)] As a $\gln$-module, $W_m(b)$ is isomorphic to $V^{\otimes m}$.
\item[(ii)] A $\glnt$-module $M$ is an irreducible subquotient of\/ $W_m(b)$
if and only if\/ $M$ has the form $L_{\bs\la}(b)$,
where\/ $\bs\la$ is a partition of\/ $m$ with at most $N$ parts.
\item[(iii)] For any natural numbers $n_1,\dots,n_k$ and distinct complex
numbers $b_1,\dots,b_k$, the $\glnt$-module $\otimes_{s=1}^k W_{n_s}(b_s)$ is
cyclic with a cyclic vector $\otimes_{s=1}^kv_{n_s}$.
\item[(iv)] Let\/ $M$ be a cyclic finite-dimensional $\glnt$-module with
a cyclic vector $v$ satisfying $e_{ij}(u)v=0$ for $1\le i<j\le N$, and
$e_{ii}(u)v=\dl_{1i}(\sum_{s=1}^k n_s(u-b_s)^{-1})v$ for $i=1,\dots,N$.
Then there exists a surjective homomorphism of $\glnt$-modules from\/
$\otimes_{s=1}^kW_{n_s}(b_s)$ to $M$ sending\/ $\otimes_{s=1}^kv_{n_s}$ to $v$.
\end{enumerate}
\end{lem}
\begin{proof}
The first two properties follow from Lemma~\ref{weyl}, parts~(ii) and~(iii).
The other two properties are proved in~\cite{CP}.
\end{proof}

Given sequences $\bs n=(n_1,\dots,n_k)$ of natural numbers and
$\bs b=(b_1,\dots,b_k)$ of distinct complex numbers, we call the $\glnt$-module
$\otimes_{s=1}^k W_{n_s}(b_s)$ the {\it Weyl module associated with\/ $\bs n$
and\/ $\bs b$}.

\begin{cor}
\label{weylbk}
A $\glnt$-module $M$ is an irreducible subquotient of\/
$\otimes_{s=1}^k W_{n_s}(b_s)$ if and only if\/ $M$ has the form
$\otimes_{s=1}^kL_{\bs\la^{(s)}}(b_s)$, where
$\bs\la^{(1)},\dots,\bs\la^{(k)}$ are partitions with at most $N$ parts
such that $|\bs\la^{(s)}|=n_s$, \,$s=1,\dots,k$.
\end{cor}
\begin{proof}
The statement follows from part (ii) of Lemma~\ref{weylb}, irreducibility
of $\otimes_{s=1}^kL_{\bs\la^{(s)}}(b_s)$ and the Jordan-H\"older theorem.
\end{proof}

Consider the $\Z_{\ge0}$-grading of the vector space $W_m$, introduced
in Lemma~\ref{weyl}. Let $W_m^j$ be the homogeneous component of $W_m$
of degree $j$ and $\bar W^j_m=\oplus_{r\ge j}W_m^{r}$.
Since the $\glnt$-module $W_m$ is graded and $W_m=W_m(b)$ as vector spaces,
$W_m(b)=\bar W_m^0\supset\bar W_m^1\supset\dots{}$ is a descending filtration
of $\glnt$-submodules. This filtration induces the structure of the associated
graded $\glnt$-module on the vector space $W_m$ which we denote by
$\gr W_m(b)$.

\begin{lem}
\label{grWmb}
The $\glnt$-module $\gr W_m(b)$ is isomorphic to the evaluation module\/
$(V^{\otimes m})(b)$.
\vvn-.3>\qed
\end{lem}

The space $\otimes_{s=1}^k W_{n_s}$ has a
$\Z_{\ge0}^k$-grading, induced by the gradings on the factors and
the associated descending $\Z_{\ge0}^k$-filtration by the subspaces
$\otimes_{s=1}^k\bar W_{n_s}^{j_s}$. This filtration is compatible with
the $\glnt$-action on the module $\otimes_{s=1}^k W_{n_s}(b_s)$. We
denote by $\gr\bigl(\otimes_{s=1}^k W_{n_s}(b_s)\bigr)$ the induced
structure of the associated graded $\glnt$-module on the space
$\otimes_{s=1}^k W_{n_s}$.

\begin{lem}
\label{grtensor}
\label{split lem}
The $\glnt$-modules \,$\gr\bigl(\otimes_{s=1}^k W_{n_s}(b_s)\bigr)$
and\/ \,$\otimes_{s=1}^k \gr W_{n_s}(b_s)$ are canonically isomorphic.
\qed
\end{lem}

\subsection{Representations of symmetric group}
\label{repsym}
Let $S_n$ be the group of permutations of $n$ elements. We denote by $\C[S_n]$
the regular representation of $S_n$. Given an $S_n$-module $M$ we denote by
$M^S$ the subspace of all $S_n$-invariant vectors in $M$.

\medskip
We need the following simple fact.
\begin{lem}\label{elementary}
Let\/ $U$ be a finite-dimensional $S_n$-module.
Then \,$\dim(U\otimes\C[S_n])^S=\dim U$.
\vvn-.3>\qed
\end{lem}

The group $S_n$ acts on the algebra
$\C[z_1,\dots,z_n]$ by permuting
the variables. Let $\si_s(\bs z)$, $s=1,\alb\dots,n$, be the $s$-th elementary
symmetric polynomial in $z_1,\dots,z_n$. The algebra of symmetric polynomials
$\C[z_1,\dots,z_n]^S$ is a free polynomial algebra with generators
$\si_1(\bs z),\dots,\alb\si_n(\bs z)$. It is well-known that the algebra
$\C[z_1,\dots,z_n]$ is a free $\C[z_1,\dots,z_n]^S$-module of rank $n!$,
see~\cite{M}.

Given $\bs a=(a_1,\dots,a_n)\in\C^n$, denote by
$I_{\bs a}\subset C[z_1,\dots,z_n]$ the ideal generated by
the polynomials $\si_s(\bs z)-a_s$, $s=1,\dots, n$.
Clearly, $I_{\bs a}$ is $S_n$-invariant.

\begin{lem}\label{regular}
For any $\bs a\in\C^n$, the $S_n$-representation $\C[z_1,\dots,z_n]/I_{\bs a}$
is isomorphic to the regular representation $\C[S_n]$.
\end{lem}
\begin{proof}
Since $\C[z_1,\dots,z_n]$ is a free $\C[z_1,\dots,z_n]^S$-module of rank $n!$,
the dimension of the quotient $\C[z_1,\dots,z_n]/I_{\bs a}$ is $n!$ for all
$\bs a\in\C^n$. If the polynomial $u^n+\sum_{s=1}^n(-1)^sa_s\>u^{n-s}$
has no multiple roots, the lemma is obvious. For other $\bs a$, the lemma
follows by the continuity of characters.
\end{proof}

\subsection{$\glnt$-module $\V^S$}
\label{VS}
Let $\V$ be the space of polynomials in $z_1,\dots,z_n$ with coefficients
in $V^{\otimes n}$:
\vvn-.3>
\be
\V\>=\,V^{\otimes n}\<\otimes_{\C}\C[z_1,\dots,z_n]\,.
\vv.2>
\ee
The space $V^{\otimes n}$ is embedded in $\V$ as the subspace of constant
polynomials.

Abusing notation, for $v\in V^{\otimes n}$ and
$p(z_1,\dots,z_n)\in\C[z_1,\dots,z_n]$, we write
$p(z_1,\dots,z_n)\,v$ instead of $v\otimes p(z_1,\dots,z_n)$.

We make the symmetric group $S_n$ act on $\V$ by permuting the factors
of $V^{\otimes n}$ and the variables $z_1,\dots,z_n$ simultaneously,
\vvn.2>
\be
\si\bigl(p(z_1,\dots,z_n)\,v_1\otimes\dots\otimes v_n\bigr)\,=\,
p(z_{\si(1)},\dots,z_{\si(n)})\,
v_{\si^{-1}(1)}\!\otimes\dots\otimes v_{\sigma^{-1}(n)}\,,\qquad\si\in S_n\,.
\kern-3em
\vv.2>
\ee
We denote by $\V^S$ the subspace of $S_n$-invariants in $\V$.

\begin{lem}
\label{VSfree}
The space $\V^S$ is a free $\C[z_1,\dots,z_n]^S$-module of rank $N^n$.
\end{lem}
\begin{proof}
The lemma follows from the fact that $\C[z_1,\dots,z_n]$ is a free
$\C[z_1,\dots,z_n]^S$-module.
\end{proof}

We consider the space $\V$ as a $\glnt$-module with
the series $g(u)$, \,$g\in\gln$, acting by
\vvn.1>
\beq
\label{action}
g(u)\,\bigl(p(z_1,\dots,z_n)\,v_1\otimes\dots\otimes v_n)\,=\,
p(z_1,\dots,z_n)\,\sum_{s=1}^n
\frac{v_1\otimes\dots\otimes gv_s\otimes\dots\otimes v_n}{u-z_s}\ .
\vv.2>
\eeq
%where the elements $e_{ij}^{(s)}=1^{\otimes(s-1)}\otimes
%e_{ij}\otimes1^{\otimes(n-s)}$ act on $V^{\otimes n}$.

\begin{lem}
\label{Uz}
The image of the subalgebra $U(\z_N[t])\subset \Uglnt$ in
$\End(\V)$ coincides with the algebra of operators of multiplication
by elements of\/ $\C[z_1,\dots,z_n]^S$.
\end{lem}
\begin{proof}
The element $\sum_{i=1}^N e_{ii}\otimes t^r$ acts on $\V$ as the operator
of multiplication by $\sum_{s=1}^n z_s^r$. The lemma follows.
\end{proof}

The $\glnt$-action on $\V$ commutes with the $S_n$-action.
Hence, $\V^S$ is a $\glnt$-submodule of $\V$.

The following lemma is contained in~\cite{K}.
For convenience we supply a proof.

\begin{lem}
\label{cycl=symm}
The $\glnt$-module $\V^S$ is cyclic with a cyclic vector\/ $v_+^{\otimes n}$.
\end{lem}
\begin{proof}
Let $v_1=v_+$ and $v_i=e_{j1}v_+$ for $j=2,\dots,N$. Then $v_1,v_2,\dots, v_N$
is a basis of $V$. Given integers $i_1\ge\dots\ge i_n\ge0$ and
$j_1,\dots,j_n \in\{1, \dots,N\}$, denote
\vvn.3>
\be
v(\bs i,\bs j)\,=\sum_{\si\in S_n}
\si (z_1^{i_1}\dots z_n^{i_n}\,v_{j_1}\otimes\dots\otimes v_{j_n})\,.
\ee
The space $\V^S$ is spanned by the elements $v(\bs i,\bs j)$ with all possible
$\bs i, \bs j$. So to prove the lemma, it is sufficient to show that every
element $v(\bs i,\bs j)$ belongs to $\Uglnt\,v_+^{\otimes n}$. We prove it
by induction on the number $r(\bs i,\bs j)$ of indices $s$ such that $i_s>0$
and $j_s=1$.

If $r(\bs i,\bs j)=0$, then $v(\bs i,\bs j)=\bigl(\,\prod_{s,\,j_s\neq 1}
e_{j_s1}\otimes t^{i_s}\bigr)\,v_+^{\otimes n}$.

If $r(\bs i,\bs j)>0$, then the difference
\be
v(\bs i,\bs j)-\Bigl(\prod_{s,\,i_s>0,\;j_s=1}e_{11}\otimes t^{i_s}\Bigr)
\Bigl(\prod_{s,\,j_s>1} e_{j_s1}\otimes t^{i_s})(v_+^{\otimes n}\Bigr)
\ee
is a linear combination of the elements $v(\bs i',\bs j')$ with
$r(\bs i',\bs j')<r(\bs i,\bs j)$ and belongs to $\Uglnt\,v_+^{\otimes n}$
by the induction hypothesis.
Therefore, $v(\bs i,\bs j)\in\Uglnt\,v_+^{\otimes n}$.
\end{proof}

Define the grading on $\C[z_1,\dots,z_n]$ by setting $\deg z_i=1$ for all
$i=1,\dots,n$. We define a grading on $\V$ by setting $\deg(v\otimes p)=\deg p$
for any $v\in V^{\otimes n}$ and any $p\in\C[z_1,\dots,z_n]$. The grading on
$\V$ induces a natural grading on $\End(\V)$.

\begin{lem}\label{cycl grad}
The $\glnt$-modules\/ $\V$ and $\V^S$ are graded.
\end{lem}
\begin{proof}
The lemma follows from formula~\Ref{action}.
\end{proof}

\subsection{Weyl modules as quotients of $\V^S$}
Let $\bs a=(a_1,\dots,a_n)\in\C^n$ be a sequence of complex numbers and
$I_{\bs a}\subset \C[z_1,\dots,z_n]$ the ideal,
defined in Section~\ref{repsym}. Define
\vvn.3>
\beq
\label{IVa}
I^\V_{\bs a}=I_{\bs a}\V^S=
\V^S\tbigcap(V^{\otimes n}\otimes I_{\bs a}).
\vv.2>
\eeq
Clearly, ${I^\V_{\bs a}}$ is a $\glnt$-submodule of $\V^S$.

Define distinct complex numbers $b_1,\dots,b_k$ and
natural numbers $n_1,\dots,n_k$ by the relation
\vvn-.6>
\beq\label{ab}
\prod_{s=1}^k\,(u-b_s)^{n_s}=\,u^n+\>\sum_{j=1}^n (-1)^j\>a_j\>u^{n-j}.
\vv-.7>
\eeq
Clearly, $\sum_{s=1}^kn_s=n$.

\begin{lem}
\label{factor=weyl}
The $\glnt$-module $\V^S/{I^\V_{\bs a}}$ is isomorphic to
the Weyl module $\otimes_{s=1}^kW_{n_s}(b_s)$.
\end{lem}
\begin{proof}
Let $\overline{v_+^{\otimes n}}$ be the image of the vector
$v_+^{\otimes n}$ in the quotient space $\V^S/{I^\V_{\bs a}}$.
Then by Lemma~\ref{cycl=symm}, the quotient space $\V^S/{I^\V_{\bs a}}$ is
a cyclic $\glnt$-module with a cyclic vector $\overline{v_+^{\otimes n}}$, and
\vvn-.3>
\be
e_{ii}(u)\,\overline{v_+^{\otimes n}}=
\sum_{s=1}^k\dl_{1i}\,\frac{n_s}{u-b_s}\ \overline{v_+^{\otimes n}}.
\ee
By part (iv) of Lemma~\ref{weylb}, there exists a surjective homomorphism
\vvn.3>
\beq
\label{homWVI}
\otimes_{s=1}^kW_{n_s}(b_s) \to \V^S/{I^\V_{\bs a}}.
\eeq
In addition,
\vvn-.5>
\begin{align*}
\dim\bigl(\V^{S}/{I^\V_{\bs a}}\bigr)\>=\>\dim (\V/{I^\V_{\bs a}})^S
&{}=\>\dim(V^{\otimes n}\otimes \C[z_1,\dots,z_n]/I_{\bs a})^S
\\[4pt]
&{}=\>\dim V^{\otimes n}=\>\dim\bigl(\otimes_{s=1}^kW_{n_s}(b_s)\bigr)\,,
\end{align*}
where we used Lemmas~\ref{regular} and~\ref{elementary} for the next to
the last equality. Therefore, the surjective homomorphism~\Ref{homWVI} is
an isomorphism.
\end{proof}

\subsection{Bethe algebra}
\label{bethesec}
Let $\der$ be the operator of formal differentiation in variable $u$.
Define the {\it universal differential operator\/} $\D^\B$ by
\vvn.3>
\be
\D^\B=\,\rdet\left( \begin{matrix}
\der-e_{11}(u) & -\>e_{21}(u)& \dots & -\>e_{N1}(u)\\
-\>e_{12}(u) &\der-e_{22}(u)& \dots & -\>e_{N2}(u)\\
\dots & \dots &\dots &\dots \\
-\>e_{1N}(u) & -\>e_{2N}(u)& \dots & \der-e_{NN}(u)
\end{matrix}\right).
\vv.2>
\ee
It is a differential operator in variable $u$, whose coefficients are
formal power series in $u^{-1}$ with coefficients in $\Uglnt$,
\vvn-.3>
\beq
\label{DB}
\D^\B=\,\der^N+\sum_{i=1}^N\,B_i(u)\,\der^{N-i}\>,
\vv-.5>
\eeq
where
\vvn-.3>
\beq
\label{Bi}
B_i(u)\,=\,\sum_{j=1}^\infty B_{ij}\>u^{-j}\,,
\eeq
and $B_{ij}\in\Uglnt$, \,$i=1,\dots,N$, \,$j\in\Z_{>0}\>$.
Clearly, $B_{ij}=0$ \>for $\,j<i$.

\begin{lem}
We have
\vvn-.4>
\beq
\label{B1}
B_1(u)\,=\,-\sum_{i=1}^N e_{ii}(u)\,,
\vv-.3>
\eeq
and
\vvn-.3>
\beq
\label{Bii}
\sum_{i=0}^N\,B_{ii}\prod_{j=0}^{N-i-1}(\al-j)\,=\,Z(\al-N+1)\,,
\vv.2>
\eeq
where $\al$ is a formal variable, $B_{00}=1$,
and\/ $\>Z(x)$ is given by formula~\Ref{Zx}.
\end{lem}
\begin{proof}
The lemma is proved by a straightforward calculation.
\end{proof}

\begin{lem}\label{B deg}
For any $i,j$, the element $B_{ij}\in\Uglnt$ is a homogeneous element of degree
$j-i$. The series $B_i(u)$ is a homogeneous series of degree $-i$.
\end{lem}
\begin{proof}
We declare the degree of $\der$ to be $-1$.
Then $\D^\B$ is homogeneous of degree $-N$. The lemma follows.
\end{proof}

We call the unital subalgebra of $\Uglnt$ generated by $B_{ij}$,
$i=1,\dots,N$, $j\in\Z_{\geq i}$, the {\it Bethe algebra\/}
and denote it by $\B$.

The next statement is established in~\cite{T}.
A polished proof can be found in~\cite{MTV1}.

\begin{thm}
The algebra $\B$ is commutative. The algebra $\B$ commutes with
the subalgebra $\Ugln\subset \Uglnt$.
\qed
\end{thm}

As a subalgebra of $\Uglnt$, the algebra $\B$ acts on any $\glnt$-module $M$.
If $K\subset M$ is a $\B$-inva\-riant subspace, then we call the image of $\B$
in $\End(K)$ the {\it Bethe algebra associated with $K$\/}. Since $\B$ commutes
with $\Ugln$, it preserves the subspace of singular vectors $M^{sing}$ as well
as weight subspaces of $M$. Therefore, the subspace $(M)_{\bs\la}^{sing}$
is $\B$-invariant for any weight $\bs\la$.

Let $\bs\la$ be a partition with at most $N$ parts. Let $n=|\bs\la|$ and
$(a_1,\dots,a_n)\in\C^n$. Define integers $k$, $n_1,\dots,n_k$ and
distinct complex numbers $b_1,\dots,b_k$ by~\Ref{ab}.
Let $\bs\La=(\bs\la^{(1)},\dots,\bs\la^{(k)})$ be a sequence of partitions
with at most $N$ parts such that $|\bs\la_s|=n_s$.

In what follows we study the action of the Bethe algebra $\B$ on the following
$\B$-modules:
\vvn.3>
\begin{align}
&\M_{\bs\la}\,=\,(\V^S)_{\bs\la}^{sing}\,,
\notag\\
&\M_\lba\,=\,(\otimes_{s=1}^k W_{n_s}(b_s))_{\bs\la}^{sing}\,, \label{MMM}
\\
&\M_{\bs\La, \bs\la,\bs b}\,=\,
(\otimes_{s=1}^kL_{\bs\la^{(s)}}(b_s))_{\bs\la}^{sing}\,.\notag
\end{align}
The $\B$-modules $\M_\lba$ and
$\M_{\bs\La, \bs\la,\bs b}$ are defined by
formula~\Ref{MMM} up to isomorphism.

We denote the Bethe algebras associated with $\M_{\bs\la}\>$,
$\>\M_\lba\>$, $\>\M_\Llb\>$ by
$\>\B_{\bs\la}\>$, $\>\B_\lba\>$,
$\>\B_\Llb\,$, respectively.

\section{Functions on Schubert cell and Wronski map}
\label{sch sec}
\subsection{Functions on Schubert cell $\Om_\lab(\infty)$}
\label{Ominfty}
Let $N,d\in\Z_{>0}$, $N\leq d$. Let $\C_d[u]$ be the space of
polynomials in $u$ of degree less than $d$. We have $\dim \C_d[u]=d$.
Let $\Gr(N,d)$ be the Grassmannian of all $N$-dimensional subspaces in
$\C[d]$. The Grassmannian $\Gr(N,d)$ is a smooth projective variety of
dimension $N(d-N)$.

For a complete flag
$\F=\{0\subset F_1\subset F_2\subset\dots\subset F_d=\C_d[u]\}$ and
a partition $\bs\la=(\la_1,\dots,\la_N)$ such that $\la_1\leq d-N$,
the {\it Schubert cell\/} \,$\Om_{{\bs\la}}(\F)\subset \Gr(N,d)$ is given by
\be
\Om_{\bs\la}(\F)=\{X\in\Gr(N,d)\ |\ \dim (X\cap F_{d-j-\la_j})=N-j\,,
\ \dim (X\cap F_{d-j-\la_j-1})=N-j-1\}\,.
\ee
We have \;$\on{codim}\,\Om_{\bs\la}(\F)=|\bs\la|$.

The Schubert cell decomposition associated to a complete flag $\F$,
see for example~\cite{GH}, is given by
\vvn-.3>
\be
\Gr(N,d)\,=\bigsqcup_{\bs\la,\,\la_1\leq d-N}\,\Om_{{\bs\la}}(\F)\,.
\ee

\sskip
Given a partition $\bs\la=(\la_1,\dots,\la_N)$ such that $\la_1\leq d-N$,
introduce a set
\vvn.2>
\be
P\,=\,\{d_1,d_2,\dots,d_N\}\,,\qquad d_i=\la_i+N-i\,,
\ee
and a new partition
\beq
\label{dual weight}
\lab\,=\,(d-N-\la_N,d-N-\la_{N-1},\dots, d-N-\la_1)\,.
\eeq

Then $d_1>d_2>\dots>d_N$ and
\be
|\bs\la|=
N(d-N)-|\lab|=\sum_{i=1}^N d_i - N(N-1)/2\,.
\ee

\medskip
Let $\F(\infty)$ be the complete flag given by
\be
\,\F(\infty)=
\{0\subset \C_1[u]\subset\C_2[u]\subset\dots\subset\C_d[u]\}\,.
\ee

We denote the Schubert cell $\Om_\lab(\F(\infty))$
by $\Om_\lab(\infty)$. We have $\dim\Om_\lab(\infty)=|\bs\la|$.

The Schubert cell $\Om_\lab(\infty)\subset\Gr(N,d)$ consists
of $N$-dimensional subspaces $X\subset\C_d[u]$ which have a basis
$\{f_1(u),\dots,f_N(u)\}$ of the form
\beq
\label{basis}
f_i(u)=u^{d_i}+\sum_{j=1,\ d_i-j\not\in P}^{d_i}f_{ij}u^{d_i-j}.
\eeq
For a given subspace $X\in\Om_\lab(\infty)$, such a basis is unique.

\sskip
Let $\O_{\bs\la}$ be the algebra of regular functions on
$\Om_\lab(\infty)$. The cell $\Om_\lab(\infty)$ is an affine space of
dimension $|\bs\la|$ with coordinate functions $f_{ij}$. Therefore, the algebra
$\O_{\bs\la}$ is a free polynomial algebra in variables $f_{ij}$,
\beq
\label{Ola}
\O_{\bs\la}=
\C[\>f_{ij}\>,\ i=1,\dots,N,\ j=1,\dots,d_i,\ d_i-j\not\in P\>].
\eeq
We often
regard the polynomials $f_i(u)$, $i=1,\dots,N$, as generating functions
for the generators $f_{ij}$ of $\O_{\bs\la}$.

Recall that the degree of $u$ is one.
Define a grading on the algebra $\O_{\bs\la}$ by setting the degree of
the generator $f_{ij}$ to be $j$. Then the generating function $f_i(u)$
is homogeneous of degree $d_i$.

\begin{lem}
\label{char O}
The graded character of $\O_{\bs\la}$ is given by the formula
\begin{gather*}
\ch(\O_{\bs\la})\,=\,\frac{\prod_{1\leq i<j\leq N}\,(1-q^{d_i-d_j})}
{\prod_{i=1}^N(q)_{d_i}}\ .
\\[-13pt]
\rightline{\qed}
\end{gather*}
\end{lem}

\subsection{Another realization of $\O_{\bs\la}$}
For $g_1,\dots,g_N \in \C[u]$, denote by
$\Wr(g_1(u),\dots,g_N(u))$ the {\it Wronskian},
\vvn-.3>
\be
\Wr(g_1(u),\dots,g_N(u))\,=\,
\det\left(\begin{matrix} g_1(u) & g_1'(u) &\dots & g_1^{(N-1)}(u) \\
g_2(u) & g_2'(u) &\dots & g_2^{(N-1)}(u) \\ \dots & \dots &\dots & \dots \\
g_N(u) & g_N'(u) &\dots & g_N^{(N-1)}(u)
\end{matrix}\right).
\ee

Let $f_i(u)$, $i=1,\dots,N$, be the generating functions given by~\Ref{basis}.
We have
\beq
\label{Wr coef}
\Wr(f_1(u),\dots,f_N(u))\,=\prod_{1\le i<j\le N}(d_j-d_i)
\ \Bigl(u^n+\sum_{s=1}^n (-1)^s\>\Sig_s\,u^{n-s}\Bigr)\,,
\eeq
where $\Sig_1,\dots,\Sig_n$ are elements of $\O_{\bs\la}$.
Define the differential operator $\D^\O_{\bs\la}$ by
\vvn.2>
\beq
\label{DOla}
\D^\O_{\bs\la}=\,\frac{1}{\Wr(f_1(u),\dots,f_N(u))}\,\rdet
\left(\begin{matrix} f_1(u) & f_1'(u) &\dots & f_1^{(N)}(u) \\
f_2(u) & f_2'(u) &\dots & f_2^{(N)}(u) \\ \dots & \dots &\dots & \dots \\
1 & \der &\dots & \der^N
\end{matrix}\right).
\eeq
It is a differential operator in variable $u$, whose coefficients are
rational functions with coefficients in $\O_{\bs\la}$,
\vvn-.4>
\beq
\label{DO}
\D^\O_{\bs\la}=\,\der^N+\sum_{i=1}^N\,F_i(u)\,\der^{N-i}\>.
\vv-.4>
\eeq
The top coefficient of the Wronskian $\Wr(f_1(u),\dots,f_N(u))$ is a constant.
Therefore, we can write
\vvn-.3>
\beq
\label{Fi}
F_i(u)\,=\,\sum_{j=1}^\infty F_{ij}\>u^{-j}\,,
\eeq
where $F_{ij}\in\O_{\bs\la}$, \,$i=1,\dots,N$, \,$j\in\Z_{>0}\>$.
Clearly, $F_{ij}=0$ for $j<i$.

\begin{lem}\label{F deg}
For any $i,j$, the element $F_{ij}\in\O_{\bs\la}$ is a homogeneous element
of degree $j-i$. The series $F_i(u)$ is a homogeneous series of degree $-i$.
\end{lem}
\begin{proof}
Recall that $\der$ has degree $-1$. So, the operator
$\D^\O_{\bs\la}$ is homogeneous of degree $-N$. The lemma follows.
\end{proof}

Define the {\it indicial polynomial of the operator
$\D^{\O\/}_{\bs\la}$ at infinity} by
\beq
\label{char at infty}
\chi(\al)\,=\,\sum_{i=0}^{N} F_{ii} \prod_{j=0}^{N-i-1} (\al-j)\,,
\eeq
where $F_{00}=1$.

\begin{lem}
We have
\vv->
\be
\chi(\al)\,=\,\prod_{i=1}^N\,(\al-d_i)\,.
\ee
\end{lem}
\begin{proof}
The coefficient of $u^{d_i-N}$ of the series $\D^\O_{\bs\la}f_i(u)$
equals $\chi(d_i)$. On the other hand, we have that
$\D^\O_{\bs\la}f_i(u)=0$, and therefore, $\chi(d_i)=0$
for all $i=1,\dots,N$. Since $\deg\chi=N$, the lemma follows.
\end{proof}

\begin{lem}\label{coef alg}
The functions $F_{ij}\in\O_{\bs\la}$, $i=1,\dots,N$,
$j=i+1,i+2,\dots$, generate the algebra $\O_{\bs\la}$.
\end{lem}
\begin{proof}
The coefficient of $u^{d_i-N-j}$ of the series
$\D^\O_{\bs\la}f_i(u)$ has the form
\vvn.3>
\be
\chi(d_i-j)\,f_{ij}+\dots{}\;,
\vv.3>
\ee
where the dots denote the terms which contain the elements $F_{kl}$ and
$f_{is}$ with $s<j$ only. Since $\D^\O_{\bs\la}f_i(u)=0$ and
$\chi(d_i-j)\ne 0$, see~\Ref{Ola}, we can express recursively the elements
$f_{ij}$ via the elements $F_{kl}$ starting with $j=1$ and then increasing
the second index $j$.
\end{proof}

\subsection{Frobenius algebras}
\label{comalg}
In this section, we recall some simple facts from commutative algebra. We use
the word {\it algebra\/} for an associative unital algebra over $\C$.

Let $A$ be a commutative algebra. The algebra $A$ considered as an $A$-module
is called the {\it regular representation \/} of $A$. The dual space $A^*$ is
naturally an $A$-module, which is called the {\it coregular representation\/}.

Clearly, the image of $A$ in $\End(A)$ for the regular representation is
a maximal commutative subalgebra. If $A$ is finite-dimensional, then the image
of $A$ in $\End(A^*)$ for the coregular representation is a maximal commutative
subalgebra as well.

If $M$ is an $A$-module and $v\in M$ is an eigenvector of the $A$-action on $M$
with eigenvalue $\xi_v\in A^*$, that is, $av=\xi_v(a)\>v$ \;for any $a\in A$,
then $\xi_v$ is a character of $A$, that is, $\xi_v(ab)=\xi_v(a)\>\xi_v(b)$.

If an element $v\in A^*$ is an eigenvector of the coregular action of $A$,
then $v$ is proportional to the character $\xi_v$. Moreover, each character
$\xi\in A^*$ is an eigenvector of the coregular action of $A$ and
the corresponding eigenvalue equals $\xi$.

A nonzero element $\xi\in A^*$ is proportional to a character if and only if
$\,\ker\xi\subset A$ is an ideal. Clearly, $A/\ker\xi\simeq\C$. On the other
hand, if $\m\subset A$ is an ideal such that $A/\m\simeq\C$, then $\m$ is
a maximal proper ideal and $\m=\ker\zeta$ for some character $\zeta$.

A commutative algebra $A$ is called {\it local\/} if it has a unique ideal $\m$
such that $A/\m\simeq\C$. In other words, a commutative algebra $A$ is local if
it has a unique character. Any proper ideal of the local finite-dimensional
algebra $A$ is contained in the ideal $\m$.

It is known that any finite-dimensional commutative algebra $A$ is isomorphic
to a direct sum of local algebras, and the local summands are in bijection
with characters of $A$.

Let $A$ be a commutative algebra. A bilinear form $(\,{,}\,):A\otimes A\to\C$
is called {\it invariant\/} if $(ab,c)=(a,bc)$ for all $a,b,c\in A$.

A finite-dimensional commutative algebra $A$ which admits an invariant
nondegenerate symmetric bilinear form ${(\,{,}\,):A\otimes A\to\C}$ is called
a {\it Frobenius algebra\/}. It is easy to see that distinct local summands of
a Frobenius algebra are orthogonal.

The following properties of Frobenius algebras will be useful.

\begin{lem}
\label{direct}
A finite direct sum of Frobenius algebras is a Frobenius algebra.
\end{lem}
\begin{proof}
Fix invariant nondegenerate symmetric bilinear forms on the summands and
define a bilinear form on the direct sum to be the direct sum of the forms
of the summands. The obtained form is clearly nondegenerate, symmetric and
invariant.
\end{proof}

Let $A$ be a Frobenius algebra. Let $I\subset A$ be a subspace.
Denote by $I^\perp\subset A$ the orthogonal complement to $I$.
Then $\dim I+\dim I^\perp=\dim A$, and the subspace $I$ is an ideal
if and only if $I^\perp$ is an ideal.

Let $A_0$ be a local Frobenius algebra with maximal ideal $\m\subset A_0$.
Then $\m^\perp$ is a one-dimen\-sional ideal. Let $m^\perp\in\m^\perp$ be
an element such that $(1,m^\perp)=1$.

\begin{lem}
\label{inverse}
%For any $a\in A_0$ there exists $b\in A_0$ such that $ab=m^\perp$.
Any nonzero ideal $I\subset A_0$ contains $\m^\perp$.
\end{lem}
\begin{proof}
Let $I\subset A_0$ be a nonzero ideal. Then $I^\perp$ is also an ideal
and $I^\perp\ne A_0$. Therefore, $I^\perp\subset\m$ and $\m^\perp\subset I$.
%
%Consider the ideal $I=aA_0$. It is nonzero, since $a\in I$.
%Hence $m^\perp\in aA_0$.
\end{proof}

For a subset $I\subset A$ define its annihilator as
$\Ann\,I\,=\,\{a\in A,\ |\ aI=0\}$. The annihilator $\Ann I$ is an ideal.

\begin{lem}
\label{Ann=perp}
Let $A$ be a Frobenius algebra and $I\subset A$ an ideal.
Then $\Ann I=I^\perp$. In particular, $\dim I+\dim\>\Ann I=\dim A$.
\end{lem}
\begin{proof}
Since every ideal in a finite-dimensional commutative algebra is a direct sum
of ideals in its local summands, it is sufficient to prove the lemma for
the case of a local Frobenius algebra.

Let $A$ be local. If $a\in I$ and $b\in\Ann I$, then $(a,b)=(ab,1)=0$.
Therefore, $\Ann I\subset I^{\perp}$.

If $b\in I^\perp$, then $bI\subset I^{\perp}$ is an ideal. If $bI\ne 0$,
then it contains $\m^{\perp}$, see Lemma~\ref{inverse}, and there exists
$a\in I$ such that $ab=m^{\perp}$. Hence, $0=(a,b)=(1,ab)=(1,m^\perp)=1$,
which is a contradiction. Therefore, $bI=0$ and $I^{\perp}\subset \Ann I$.
\end{proof}

For any ideal $I\subset A$, the regular action of $A$ on itself induces
an action of $A/I$ on $\Ann I$.

\begin{lem}
\label{coreg}
The $A/I$-module $\Ann I$ is isomorphic to the coregular representation of
$A/I$. In particular, the image of $A/I$ in $\End(\Ann I)$ is a maximal
commutative subalgebra.
\end{lem}
\begin{proof}
The form $(\,{,}\,)$ gives a natural isomorphism of the ideal $\Ann I$ and
the dual space $(A/I)^*$, which identifies the action of $A/I$ on $\Ann I$
with the coregular action of $A/I$ on $(A/I)^*$.
\end{proof}

Let $P_1,\dots,P_m$ be polynomials in variables $x_1,\dots,x_m$.
Denote by $I$ the ideal in $\C[x_1,\dots,x_m]\kern-1em$ \kern1em
generated by $P_1,\dots,P_m$.

\begin{lem}
\label{resform}
If the algebra $\C[x_1,\dots,x_m]/I$ is nonzero and finite-dimensional,
then it is a Frobenius algebra.
\end{lem}
\begin{proof}
For $F,G\in\C[x_1,\dots,x_m]$, define the {\it residue form\/}
\vvn.3>
\be
\Res(F,G)=\frac{1}{(2\pi i)^m}\int_\Gm
\frac{FG\ dx_1\wedge\dots\wedge dx_m}{\prod_{s=1}^m P_s(x_1,\dots,x_m)}\ ,
\vv.3>
\ee
where $\Gm=\{(x_1,\dots,x_m)\ |\ |P_s(x_1,\dots,x_m)|=\eps\}$ is
\vvn.3>
the real $m$-cycle oriented by the condition
\be
d\arg P_1(x_1,\dots,x_m)\wedge\dots\wedge d\arg P_m(x_1,\dots,x_m)\ge0\,
\vv.3>
\ee
and $\eps$ is a small positive real number.
The residue form $\>\Res$ descends to a nondegenerate bilinear form on
$\C[x_1,\dots,x_m]/I$, see~\cite{GH}.
\end{proof}

The last lemma has the following generalization. Let $\C_T(x_1,\dots,x_m)$
be the algebra of rational functions in $x_1,\dots,x_m$, regular at points of a
nonempty set $T\subset\C^m$. Denote by $I_T$ the ideal in $\C_T(x_1,\dots,x_m)$
generated by $P_1,\dots,P_m$.

\begin{lem}
\label{resformloc}
If the algebra $\C_T(x_1,\dots,x_m)/I_T$ is nonzero and finite-dimensional,
then it is a Frobenius algebra.
\qed
\end{lem}

\subsection{Wronski map}
\label{wronski}
Let $X$ be a point of $\Gr(N,d)$.
The Wronskian of a basis of the subspace $X$ does not depend on the choice
of the basis up to multiplication by a nonzero number. We call the monic
polynomial representing the Wronskian {\it the Wronskian of\/} $X$ and
denote it by $\Wr_X(u)$.

Fix a partition $\bs\la$ and denote $n=|\bs\la|$. The partition $\lab$
is given by~\Ref{dual weight}. If $X\in\Om_\lab(\infty)$, then
$\deg\Wr_X(u)=n$.

Define the {\it Wronski map\/}
\vvn-.3>
\be
\label{wronski map}
\Wrl:\Om_\lab(\infty)\to\C^n,
\vv.3>
\ee
by sending $X\in \Om_\lab(\infty)$ to $\bs a=(a_1,\dots,a_n)$,
if $\Wr_X(u)\>=\>u^n+\sum_{s=1}^n (-1)^sa_s u^{n-s}$.

For $\bs a\in\C^n$, let ${I^\O_\lba}$ be the ideal
in $\O_{\bs\la}\>$ generated by the elements $\Sig_s-a_s$, $s=1,\dots, n$,
where $\Sig_1,\dots,\Sig_n$ are defined by~\Ref{Wr coef}. Let
\vvn.3>
\beq
\label{Olaa}
\O_\lba\>=\,\O_{\bs\la}/I^\O_\lba
\vv.3>
\eeq
be the quotient algebra. The algebra $\O_\lba$
is the scheme-theoretic fiber of the Wronski map. We call it
the {\it algebra of functions on the preimage\/} $\Wrli(\bs a)$.

\goodbreak
\begin{lem}\label{local Wr}
The algebra $\O_\lba$ is a finite-dimensional
Frobenius algebra and $\dim_\C\,\O_\lba$
does not depend on $\bs a$.
\end{lem}
\begin{proof}
It is easy to show that the set $\Wrli(\bs a)$ is finite.
This implies that $\O_\lba$ is finite-dimensional and the fact that $\O_\lba$
is a direct sum of local algebras, see for example~\cite{HP}.
The dimension of $\O_\lba$ is the degree of the Wronski map and
the local summands correspond to the points of the set $\Wrli(\bs a)$.

The algebra $\O_\lba$ is Frobenius by Lemma~\ref{resform}.
\end{proof}

\begin{rem}
Let $\O_{\bs\la}^S\subset\O_{\bs\la}$ be the subalgebra generated by
$\Sig_1,\dots,\Sig_n$. Since these elements are homogeneous, the subalgebra
$\O_{\bs\la}^S$ is graded. Using the grading and
Lemmas~\ref{local Wr}, \ref{weyl}, it is easy to see that
$\O_{\bs\la}$ is a free $\O_{\bs\la}^S$-module
of rank $\dim\>(V^{\otimes n})^{sing}_{\bs\la}$.
\end{rem}

\section{Functions on intersection of Schubert cells}
\label{schubert}
\subsection{Functions on $\Om_\Lbl$}
For $b\in\C$, consider the complete flag
\vvn.3>
\be
\F(b)\,=\,\bigl\{0\subset (u-b)^{d-1}\>\C_1[u]\subset(u-b)^{d-2}\>\C_2[u]
\subset\dots\subset\C_d[u]\bigr\}\,.
\vv.3>
\ee
We denote the Schubert cell $\Om_{\bs\mu}\bigl(\F(b)\bigr)$ corresponding
to the flag $\F(b)$ and a partition $\bs\mu$ with at most $N$ parts by
$\Om_{\bs\mu}(b)$.

Let $\bs\La=(\bs\la^{(1)},\dots,\bs\la^{(k)})$ be a sequence of partitions
with at most $N$ parts such that $\sum_{s=1}^k|\bs\la^{(s)}|=n$.
Denote $n_s=|\bs\la^{(s)}|$.
Let $\bs b=(b_1,\dots,b_k)$ be a sequence of distinct complex numbers.

Denote by $\Om_\Lbl$ the intersection of the Schubert cells:
\beq
\label{Omega}
\Om_\Lbl\,=\,\Om_\lab(\infty)\;\cap\;
\bigcap_{s=1}^k\Om_{\bs\la^{(s)}}(b_s)\,,
\eeq
where the cell $\Om_\lab(\infty)$ is defined in Section~\ref{Ominfty}.
Recall that a subspace $X\subset\C_d[u]$ is a point of $\Om_\lab(\infty)$
if and only if for every $i=1,\dots,N$, it contains a monic polynomial
of degree $d_i$. Similarly, the subspace $X$ is a point of
$\Om_{\bs\la^{(s)}}(b_s)$ if and only if for every $i=1,\dots,N$,
it contains a polynomial with a root at $b_s$ of order $\la_i^{(s)}+N-i$.

Given an $N$-dimensional space of polynomials $X\subset\C[u]$, denote by
$\D_X$ the monic scalar differential operator of order $N$ with kernel $X$.
The operator $\D_X$ is a Fuchsian differential operator.
If $X\in\Om_\lab(\infty)$, then $\D_X$ equals the operator
$\D^\O_{\bs\la}$, see~\Ref{DOla}, computed at $X$.

\begin{lem}
\label{lem on intersection}
A subspace $X\subset\C_d[u]$ is a point of\/ $\Om_\Lbl$ if and only if
the singular points of the operator\/ $\D_X$ are at $b_1,\dots,b_k$ and
$\infty$ only, the exponents at\/ $b_s$, $s=1,\dots,k$, being equal to
$\la_N^{(s)},\,\la_{N-1}^{(s)}+1,\,\dots\,,\la_1^{(s)}+N-1$, and the exponents
at $\infty$ being equal to $1-N-\la_1,\,2-N-\la_2,\,\dots\,,-\>\la_N$.
\qed
\end{lem}

\begin{lem}
\label{lem inclus}
Let $\bs a=(a_1,\dots,a_n),\,\bs b=(b_1,\dots,b_k)$, and\/ $n_1,\dots,n_k$
be related as in~\Ref{ab}.
Then\/ $\Om_\Lbl\subset\Wrli(\bs a)$.
In particular, the set $\,\Om_\Lbl$ is finite.
\qed
\end{lem}

Let $\Q_{\bs\la}$ be the field of fractions of $\O_{\bs\la}\>$,
and $\>\Q_\Llb\subset\Q_{\bs\la}\,$ the subring of elements
regular at all points of $\Om_\Lbl$.

Consider the $N\,{\times}\,N$ matrices $M_1,\dots,M_k$ with entries in
$\O_{\bs\la}$\>,
\be
(M_s)_{ij}\>=\,\frac1{(\la^{(s)}_j+N-j)!}\,\biggl(\<\Bigl(
\frac d{du}\Bigr)^{\la^{(s)}_j+N-j}f_i(u)\biggr)\bigg|_{u=b_s}\,.
\ee
The values of $M_1,\dots,M_k$ at any point of $\Om_\Lbl$
are matrices invertible over $\C$. Therefore, the inverse matrices
$M_1^{-1},\dots,M_k^{-1}$ exist as matrices with entries in
$\Q_\Llb$.

\sskip
Introduce the elements $g_{ijs}\in\Q_\Llb$\>, \,$i=1,\dots,N$,
$j=0,\dots,d_1$, $s=1,\dots,k$, by the rule
\beq
\label{gijs}
\sum_{j=0}^{d_1}\,g_{ijs}\,(u-b_s)^j\,=\,
\sum_{m=1}^N\,(M_s^{-1})_{im}\,f_m(u)\,.
\eeq
Clearly, $g_{i,\la^{(s)}_j+N-j,s}=\dl_{ij}$ for all $i,j=1,\dots,N$, and
$s=1,\dots,k$.

For each $s=1,\dots,k$, let $J^{\Q,s}_\Llb$ be
the ideal in $\Q_\Llb\>$ generated by the elements $g_{ijs}$,
$\,i=1,\dots,N$, $\,j=0,1,\dots,\la^{(s)}_i+N-i-1$, and
$J^\Q_\Llb=\sum_{s=1}^k J^{\Q,s}_\Llb$.
Note that the number of generators of the ideal
$J^\Q_\Llb$ equals $n$.

\sskip
The quotient algebra
\vvn-.4>
\beq
\label{OX}
\O_\Llb\,=\,\Q_\Llb/J^\Q_\Llb
\eeq
is the scheme-theoretic intersection of the Schubert cells. We call it the
{\it algebra of functions on\/} $\Om_\Lbl\,$.

\begin{lem}
\label{frobenius}
The algebra $\O_\Llb$ is a Frobenius algebra.
\end{lem}
\begin{proof}
The lemma follows from Lemma~\ref{resform}.
\end{proof}

It is known from Schubert calculus that
\beq
\label{dimO}
\dim\O_\Llb\,=\,\dim\>
(\otimes_{s=1}^kL_{\bs\la^{(s)}})_{\bs\la}^{sing}\,,
\eeq
for example, see~\cite{Fu}.

\subsection{Algebra $\O_\Llb$ as quotient of $\O_{\bs\la}$}
Consider the differential operator
\beq
\label{DOlat}
\Dt^\O_{\bs\la}=\>\Wr(f_1,\dots,f_N)\D^\O_{\bs\la}=\,\rdet
\left(\begin{matrix} f_1(u) & f_1'(u) &\dots & f_1^{(N)}(u) \\
f_2(u) & f_2'(u) &\dots & f_2^{(N)}(u) \\ \dots & \dots &\dots & \dots \\
1 & \der &\dots & \der^N
\end{matrix}\right).
\eeq
It is a differential operator in variable $u$ whose coefficients are
polynomials in $u$ with coefficients in $\O_{\bs\la}$,
\vvn-.6>
\beq
\label{DOt}
\Dt^\O_{\bs\la}=\,\sum_{i=0}^N\,G_i(u)\,\der^{N-i}\>.
\eeq
Clearly, if $n-i<0$, then $G_i(u)=0$, otherwise $\deg G_i=n-i$ \,and
\vvn.3>
\be
G_i(u)\,{}=\,\Wr(f_1(u),\dots,f_N(u))\>F_i(u),
\ee
where $i=0,\dots,N$, and $F_0(u)=1$.

Introduce the elements $G_{ijs}\in\O_{\bs\la}$\>, \,$i=0,\dots,N$,
$j=0,1,\dots,n-i$, $s=1,\dots,k$, by the rule
\beq
\label{Gijs}
G_i(u)\,=\,\sum_{j=0}^{n-i}\,G_{ijs}\,(u-b_s)^j\,.
\eeq
We set $G_{ijs}=0$ if $j<0$ or $j>n-i$.

Define the {\it indicial polynomial\/ $\chi_s^\O(\al)$ at\/ $b_s$}
by the formula
\be
\chi_s^\O(\al)\,=\,
\sum_{i=0}^{N}\,G_{i,n_s-i,s}\prod_{j=0}^{N-i-1} (\al-j)\,.
\ee
It is a polynomial of degree $N$ in variable $\al$ with coefficients in
$\O_{\bs\la}$.

\begin{lem}
\label{chis}
For a complex number $r$, the element\/ $\chi_s^\O(r)$ is invertible
in\/ $\Q_\Llb$ provided\/ $r\ne\la^{(s)}_j+N-j$ \,for all\/ $j=1,\dots,N$.
\end{lem}
\begin{proof}
An element of $\Q_\Llb$ is invertible if and only if its value
at any point of $\Om_\Lbl$ is nonzero. Now the claim follows from
Lemmas~\ref{lem on intersection} and~\ref{lem inclus}.
\end{proof}

For each $s=1,\dots,k$, let $I^{\Q,s}_\Llb$ be
the ideal in $\Q_\Llb\>$ generated by the elements $G_{ijs}$,
$\,i=0,\dots,N$, $\,j=0,1,\dots,n_s-i-1$,
and the coefficients of the polynomials
\beq
\label{chisO}
\chi_s^\O(\al)\,-\,\prod_{\fratop{r=1}{r\ne s}}^k\,(b_s-b_r)^{n_r}\,
\prod_{l=1}^N\,(\al-\la_l^{(s)}-N+l)\,,\qquad s=1,\dots,k\,.\kern-3em
\eeq
Denote \,$I^\Q_\Llb=\sum_{s=1}^k I^{\Q,s}_\Llb\,$.

\begin{lem}
\label{ideals}
For any $s=1,\dots,k$, the ideals\/ $I^{\Q,s}_\Llb$ and\/ $J^{\Q,s}_\Llb$
coincide.
\end{lem}
\begin{proof}
Consider the differential operator $\Dt^\O_{\bs\la}$, given by~\Ref{DOlat}.
We have
\beq
\label{DOGs}
\Dt^\O_{\bs\la}\>=\,
\sum_{i=0}^N\,\sum_{j=0}^{n-i}\,G_{ijs}\,(u-b_s)^j\,\der^{N-i}\,,
\vv-.1>
\eeq
see~\Ref{DO}, \Ref{Gijs}, and
\vvn-.5>
\beq
\label{DOgs}
\Dt^\O_{\bs\la}\>=\,\det M_s\cdot\>
\rdet\left(\begin{matrix} g_{1s}(u) & g_{1s}'(u) &\dots & g_{1s}^{(N)}(u) \\
g_{2s}(u) & g_{2s}'(u) &\dots & g_{2s}^{(N)}(u) \\ \dots&\dots&\dots&\dots \\
1 & \der &\dots & \der^N
\end{matrix}\right),
\eeq
where $g_{is}(u)=\sum_{j=0}^{d_1}\,g_{ijs}\,(u-b_s)^j$,
see~\Ref{DOla}, \Ref{gijs}. Formulae~\Ref{DOGs} and~\Ref{DOgs} imply that
$I^{\Q,s}_\Llb\subset J^{\Q,s}_\Llb$\,.

To get the opposite inclusion, $J^{\Q,s}_\Llb\subset I^{\Q,s}_\Llb$\,,
write \>$\Dt^\O_{\bs\la}\>g_{is}(u)=\sum_{j=0}^{n-N}\,q_{ijs}\,(u-b_s)^j$.
Then
\vvn-.3>
\beq
\label{manydots}
q_{i,n_s-N+r,s}\,=\,\sum_{j=0}^{d_1}\,\sum_{l=0}^N\,
g_{ijs}\,G_{l,n_s+r-j-l,s}\,\prod_{m=0}^{N-l-1}(j-m)\,,
\eeq
where it is assumed that the elements $\>g_{ijs}\>$ and $\>G_{l,n_s+r-j-l,s}\>$
equal zero if their subscripts are out of range of definition.
Observe that the terms with $j=r$ in the right-hand side of~\Ref{manydots}
sum up to $\chi_s^\O(r)\,g_{irs}$, and the terms with $j>r$ belong to the ideal
$I^{\Q,s}_\Llb$.

We have \>$\Dt^\O_{\bs\la}\>g_{is}(u)=0$, so that \,$q_{i,n_s-N+r,s}=0$.
Using Lemma~\ref{chis} and taking into account that $g_{i,\la^{(s)}_l+N-l,s}=0$
for $l>r$, we can show recursively that the elements $g_{irs}$ with
$r<\la^{(s)}_i+N-i$ \>belong to $I^{\Q,s}_\Llb$\,,
starting with $r=0$ and then increasing the second index $r$.
Therefore, $J^{\Q,s}_\Llb\subset I^{\Q,s}_\Llb$\,.
\end{proof}

Let $I^\O_\Llb$ be the ideal in $\O_{\bs\la}\>$ generated by the elements
$G_{ijs}$, $\,i=0,\dots,N$, \>$s=1,\dots,k$, $\,j=0,1,\dots,n_s-i-1$, and
the coefficients of polynomials~\Ref{chisO}.

\begin{prop}
\label{OIX}
The algebra \>$\O_\Llb$ is isomorphic to the quotient algebra\/
\>$\O_{\bs\la}/I^\O_\Llb\>$.
\end{prop}
\begin{proof}
By Lemma~\ref{ideals}, the ideals\/ $I^\Q_\Llb$ and\/ $J^\Q_\Llb$ coincide,
so the algebra \>$\O_\Llb$ is isomorphic to the quotient algebra
\>$\Q_\Llb/I^\Q_\Llb\>$. By Lemma~\ref{lem on intersection}, the algebraic
set defined by the ideal $I^\O_\Llb$ equals $\Om_\Lbl$\,. The set $\Om_\Lbl$
is finite by Lemma~\ref{lem inclus}. Therefore, the quotient algebras
$\Q_\Llb/I^\Q_\Llb\>$ and $\O_{\bs\la}/I^\O_\Llb\>$ are isomorphic.
\end{proof}

\subsection{Algebra $\O_\Llb$ as quotient of $\O_\lba$}
\label{barI}
Recall that $\O_\lba=\O_{\bs\la}/{I^\O_\lba}$ is the algebra of functions on
$\Wrli(\bs a)$, see~\Ref{Olaa}. For an element $F\in\O_{\bs\la}$,
we denote by $\bat F$ the projection of $F$ to the quotient algebra $\O_\lba$.

\begin{lem}
\label{nil1}
The elements $\bat G_{ijs}$, $\,i=1,\dots,N$, $\,s=1,\dots,k$,
$\,j=0,1,\dots,n_s-i-1$, are nilpotent.
\end{lem}
\begin{proof}
The values of these elements on every element $X\in \Wr^{-1}(\bs a)$
are clearly zero. The lemma follows.
\end{proof}

Define the {\it indicial polynomial\/ $\bat\chi_s^\O(\al)$ at\/ $b_s$}
by the formula
\be
\bat\chi_s^\O(\al)\,=\,
\sum_{i=0}^{N}\,\bat G_{i,n_s-i,s}\prod_{j=0}^{N-i-1}(\al-j)\,.
\ee
Let $\bat I^\O_\Llb$ be the ideal in $\O_\lba\>$ generated by the elements
$\bat G_{ijs}$, $\,i=1,\dots,N$, $\,s=1,\dots,k$, $j=0,1,\dots, n_s-i-1$,
and the coefficients of the polynomials
\be
\bat\chi_s^\O(\al)\,-\,\prod_{\fratop{r=1}{r\ne s}}^k\,(b_s-b_r)^{n_r}\,
\prod_{l=1}^N\,(\al-\la_l^{(s)}-N+l)\,,\qquad s=1,\dots,k\,.\kern-3em
\ee

\begin{prop}
\label{scheme}
The algebra\/ $\O_\Llb$ is isomorphic to the quotient
algebra\/ $\O_\lba/\bat I^\O_\Llb$.
\end{prop}
\begin{proof}
The elements $G_{0js}$, $j=0,\dots,n_s-1$,
$s=1,\dots,k$, generate the ideal $I^\O_\lba$ in $\O_{\bs\la}$\>. Moreover,
the projection of the ideal $I^\O_\Llb\subset\O_{\bs\la}$ \>to $\O_\lba$ equals
$\bat I^\O_\Llb$\,. Hence, the proposition follows from Proposition~\ref{OIX}.
\end{proof}

\begin{example}
Let $N=2$, $n=3$, $d>4$, $\bs\la = (2,1)$, $\bs a =(0,0,0)$.
Then $k=1$, $\bs b=(b_1)$ with $b_1=0$, and $n_1=3$, and we have
\begin{gather*}
f_1(u) = u^3 + f_{11}\>u^2 + f_{13}\,,\qquad f_2(u)=u+f_{21}\,,
\\[6pt]
\Wr(f_1(u),f_2(u))\,=\,-2u^3-(f_{11}+3f_{21})u^2 - 2f_{11}f_{21}u + f_{13}\,,
\\[6pt]
\O_\lba\,=\,\C[f_{11},f_{13},f_{21}]\big/
\langle (f_{11} + 3f_{21}), 2f_{11}f_{21}, f_{13} \rangle\,,
\end{gather*}
so the algebra $\O_\lba$ equals $\C\>1+\C\bat f_{11}$ with
$\bat f_{11}^2=0$.

If $\bs\La=(\bs\la^{(1)})$ with $\bs\la^{(1)}=(2,1)$, then the ideal
$\bat I^\O_\Llb$ is generated by the element $\bat f_{11}$ and
$\dim\O_\lba/\bat I^\O_\Llb=1$.

If $\bs\La=(\bs\la^{(1)})$ with $\bs\la^{(1)}=(3,0)$, then the ideal
$\bat I^\O_\Llb$ is generated by the elements $3$ and $\bat f_{11}$,
and $\dim\O_\lba/\bat I^\O_\Llb=0$.
\end{example}

Recall that the ideal $\Ann(\bat I^\O_\Llb)\subset\O_\lba$ is naturally
an $\O_\Llb$-module.

\begin{cor}
\label{AnnI}
The\/ $\O_\Llb$-module\/ $\Ann(\bat I^\O_\Llb)$ is isomorphic to the coregular
representation of $\O_\Llb$ on the dual space $(\O_\Llb)^*$.
\end{cor}
\begin{proof}
The statement follows from Lemmas~\ref{frobenius} and~\ref{coreg}.
\end{proof}

\section{Three isomorphisms}
\label{iso sec}
\subsection{Case of generic $\bs a$}
Fix natural numbers $n$ and $d\ge N$, and a partition
$\bs\la=(\la_1,\dots,\la_N)$ such that $|\bs\la|=n$ and $\la_1\leq d-N$.
Define the partition $\lab$ by formula~\Ref{dual weight}.

Recall that given an $N$-dimensional space of polynomials $X\subset\C[u]$,
we denote by $\D_X$ the monic scalar differential operator of order $N$ with
kernel $X$.

Let $M$ be a $\glnt$-module $M$ and $v$ an eigenvector of the Bethe algebra
$\B\subset\Uglnt$ acting on $M$. Then for any coefficient $B_i(u)$ of
the universal differential operator $\D^\B$ we have $B_i(u)v=h_i(u)v$,
where $h_i(u)$ is a scalar series. We call the scalar differential operator
\beq
\label{DBv}
\D^\B_v\>=\,\der^N+\>\sum_{i=1}^N\>h_i(u)\,\der^{N-i}
\eeq
{\it the differential operator associated with\/} $v$.

\begin{lem}
\label{generic}
There exist a Zariski open $S_n$-invariant subset $\Tht$ of $\C^n$ and
a Zariski open subset $\Xi$ of $\Om_\lab(\infty)$ with the following
properties.
\begin{enumerate}
\item[(i)]
For any $(b_1,\dots,b_n)\in\Tht$, all numbers $b_1,\dots,b_n$ are distinct,
and there exists a basis of
$\bigl(\otimes_{s=1}^nV(b_s)\bigr)_{\bs\la}^{sing}$ such that every basis
vector $v$ is an eigenvector of the Bethe algebra and $\D^\B_v=\D_X$ for some
$X\in\Xi$. Moreover, different basis vectors correspond to different points
of\/ $\Xi$.
\item[(ii)]
For any $X\in\Xi$\,, if\/ $b_1,\dots,b_n$ are all roots of the Wronskian\/
$\Wr_X$, then $(b_1,\dots,b_n)\in\Tht$, and there exists a unique up to
proportionality vector $v\in\bigl(\otimes_{s=1}^nV(b_s)\bigr)_{\bs\la}^{sing}$
such that $v$ is an eigenvector of the Bethe algebra and $\D^\B_v=\D_X$.
\end{enumerate}
\end{lem}
\begin{proof}
The basis in part (i) is constructed by the Bethe ansatz method,
see~\cite{MV2}. The equality $\D^\B_v=\D_X$ is proved in~\cite{MTV1}.
The existence of an eigenvector $v$ in part (ii) for generic
$X\subset\Om_\lab$, is proved in~\cite{MV1}.
\end{proof}

For generic $\bs a=(a_1,\dots,a_n)$, the roots $(b_1,\dots,b_n)$ of
the polynomial $\sum_{i=0}^na_ix^i$ form a point in $\Theta$. Therefore,
Lemma~\ref{generic}, in particular, asserts that for generic values of
$\bs a\in\C^n$, the algebras $\B_\lba$ and $\O_\lba$
are both isomorphic to the direct sum of
$\dim\>(V^{\otimes n})_{\bs\la}^{sing}$ copies of $\C$.

The following corollary recovers a well-known fact.
\begin{cor}\label{degree}
The degree of the Wronski map equals\/
\,$\dim\>(V^{\otimes n})_{\bs\la}^{sing}$.
\qed
\end{cor}

\subsection{Isomorphism of algebras $\O_{\bs\la}$ and $\B_{\bs\la}$}
In this section we show that the Bethe algebra $\B_{\bs\la}$, associated
with the space $\M_{\bs\la}=(\V^S)_{\bs\la}^{sing}$, see~\Ref{MMM},
is isomorphic to the algebra $\O_{\bs\la}$ of regular functions of
the Schubert cell $\Om_\lab(\infty)$, and that under this isomorphism
the $\B_{\bs\la}$-module $\M_{\bs\la}$ is isomorphic to
the regular representation of $\O_{\bs\la}$.

Consider the map
\vvn-.1>
\be
\tau_{\bs\la}:\O_{\bs\la}\to\B_{\bs\la}\,,\qquad
F_{ij}\mapsto\hat B_{ij}\,,
\vv.2>
\ee
where the elements $F_{ij}\in\O_{\bs\la}$ are defined by~\Ref{Fi} and
\>$\hat B_{ij}\in\B_{\bs\la}$ are the images of the elements $B_{ij}\in\B$,
defined by~\Ref{Bi}.

\begin{thm}
\label{first}
The map $\tau_{\bs\la}$ is a well-defined isomorphism of graded algebras.
\end{thm}
\begin{proof}
Let a polynomial $R(F_{ij})$ in generators $F_{ij}$ be equal to zero
in $\O_{\bs\la}$. Let us prove that the corresponding polynomial
$R(\hat B_{ij})$ is equal to zero in the $\B_{\bs\la}$. Indeed,
$R(\hat B_{ij})$ is a polynomial in $z_1,\dots,z_n$ with values in
$\End\bigl((V^{\otimes n})_{\bs\la}^{sing}\bigr)$. Let $\Tht$ be the set,
introduced in Lemma~\ref{generic}, and $(b_1,\dots,b_n)\in\Tht$.
Then by part (i) of Lemma~\ref{generic}, the value of the polynomial
$R(\hat B_{ij})$ at $z_1=b_1,\dots,z_n=b_n$ equals zero. Hence, the polynomial
$R(\hat B_{ij})$ equals zero identically and the map $\tau_{\bs\la}$ is
well-defined.

By Lemmas~\ref{F deg} and~\ref{B deg} the elements $F_{ij}$ and $\hat B_{ij}$
are of the same degree. Therefore, the map $\tau_{\bs\la}$ is graded.

Let a polynomial $R(F_{ij})$ in generators $F_{ij}$ be a nonzero element of
$\O_{\bs\la}$. Then the value of $R(F_{ij})$ at a generic point
$X\in\Om_\lab(\infty)$ is not equal to zero. Then by part (ii) of
Lemma~\ref{generic}, the polynomial $R(\hat B_{ij})$ is not identically equal
to zero. Therefore, the map $\tau_{\bs\la}$ is injective.

Since the elements $\hat B_{ij}$ generate the algebra $\B_{\bs\la}\,$,
the map $\tau_{\bs\la}$ is surjective.
\end{proof}

The algebra $\C[z_1,\dots,z_n]^S$ is embedded into the algebra $\B_{\bs\la}$
as the subalgebra of operators of multiplication by symmetric polynomials,
see Lemmas~\ref{Uz} and formula~\Ref{Bii}. The algebra $\C[z_1,\dots,z_n]^S$
is embedded into the algebra $\O_{\bs\la}$\>, the elementary symmetric
polynomials $\si_1(\bs z),\dots,\si_n(\bs z)$ being mapped to the elements
$\Sig_1,\dots,\Sig_n$, defined by~\Ref{Wr coef}. These embeddings give
the algebras $\B_{\bs\la}$ and $\O_{\bs\la}$ the structure of
$\C[z_1,\dots,z_n]^S\<$-modules.

\begin{lem}\label{symm OK}
We have $\tau_{\bs\la}(\Sig_i)=\si_i(\bs z)$ \>for all\/ \>$i=1,\dots,n$.
In particular, the map $\,\tau_{\bs\la}:\O_{\bs\la}\to\B_{\bs\la}$ is
an isomorphism of $\,\C[z_1,\dots,z_n]^S\<$-modules.
\end{lem}
\begin{proof}
The lemma follows from the fact that
\beq
\label{F1}
F_1(u)\,=\,-\,\frac{\Wr'(f_1(u),\dots,f_N(u))}{\Wr(f_1(u),\dots,f_N(u))}\ ,
\eeq
where $\,'$ denotes the derivative with respect to $u$, and formula~\Ref{B1}.
\end{proof}

Given a vector $v\in\M_{\bs\la}$, consider a linear map
\be
\mu_v:\O_{\bs\la}\to\M_{\bs\la}\,,\qquad F\mapsto\tau_{\bs\la}(F)\,v\,.
\ee

\begin{lem}
\label{muinject}
If\/ $v\in\M_{\bs\la}$ is nonzero, then the map \>$\mu_v$ is injective.
\end{lem}
\begin{proof}
The algebra $\O_{\bs\la}$ is a free polynomial algebra containing
the subalgebra $\C[z_1,\dots,z_n]^S$. By Lemma~\ref{local Wr}, the
quotient algebra $\O_{\bs\la}/\C[z_1,\dots,z_n]^S$ is finite-dimensional.
The kernel of $\mu_v$ is an ideal in $\B_{\bs\la}$ which has zero
intersection with $\C[z_1,\dots,z_n]^S$ and, therefore, is the zero ideal.
\end{proof}

By Lemmas~\ref{VSfree} and~\ref{cycl grad}, the space $\M_{\bs\la}$ is
a free graded $\C[z_1,\dots,z_n]^S$-module. By Lemma~\ref{factor=weyl},
the generators of this module can be identified with a basis in
$(W_n)_{\bs\la}^{sing}$. Therefore, the graded character of $\M_{\bs\la}$
is given by the formula
\beq
\label{char cycl}
\ch(\M_{\bs\la})\,=\,\frac{\ch\bigl((W_n)_{\bs\la}^{sing}\bigr)}{(q)_n}\ .
\eeq
This equality and part~(iv) of Lemma~\ref{weyl} imply that the degree of
any vector in $\M_{\bs\la}$ is at least $\sum_{i=1}^N(i-1)\la_i$ and the
homogeneous component of $\M_{\bs\la}$ of degree $\sum_{i=1}^N(i-1)\la_i$ is
one-dimensional.

Fix a nonzero vector $v\in\M_{\bs\la}$ of degree $\sum_{i=1}^N(i-1)\la_i$.
Denote the map $\mu_v$ by $\mu_{\bs\la}$.

\begin{thm}
\label{first1}
The map\/ \>$\mu_{\bs\la}:\B_{\bs\la}\to\M_{\bs\la}$ is an isomorphism
of degree $\sum_{i=1}^N(i-1)\la_i$ of graded vector spaces.
For any $F,G\in\O_{\bs\la}$, we have
\beq
\label{mutau}
\mu_{\bs\la}(FG)\,=\,\tau_{\bs\la}(F)\,\mu_{\bs\la}(G)\,.
\eeq
In other words, the maps\/ $\tau_{\bs\la}$ and\/ $\mu_{\bs\la}$ give
an isomorphism of the regular representation of\/ $\O_{\bs\la}$ and
the\/ $\B_{\bs\la}$-module $\M_{\bs\la}$.
\end{thm}
\begin{proof}
The map $\mu_{\bs\la}$ is injective by Lemma~\ref{muinject}.
The map $\mu_{\bs\la}$ shifts the degree by $\sum_{i=1}^N(i-1)\la_i$.
Lemma~\ref{char O}, formula~\Ref{char cycl}, and part~(iv) of Lemma~\ref{weyl}
imply that $\ch\bigl(\mu_{\bs\la}(\O_{\bs\la})\bigr)=\ch(\M_{\bs\la})$.
Hence, the map $\mu_{\bs\la}$ is surjective.
Formula~\Ref{mutau} follows from Theorem~\ref{first}.
\end{proof}

\subsection{Isomorphism of algebras $\O_\lba$ and $\B_\lba$}
Let $\bs a=(a_1,\dots,a_N)$ be a sequence of complex numbers.
Let distinct complex numbers $b_1,\dots,b_k$ and integers $n_1,\dots,n_k$
be given by~\Ref{ab}.

In this section we show that the Bethe algebra $\B_\lba$\,, associated with the
space $\M_\lba=(\otimes_{s=1}^kW_{n_s}(b_s))_{\bs\la}^{sing}$, see~\Ref{MMM},
\vvn.1>
is isomorphic to the algebra $\O_\lba$ of functions on the preimage
$\Wrli(\bs a)$. We also show that under this isomorphism the $\B_\lba$-module
$\M_\lba$ is isomorphic to the regular representation of $\O_\lba$.

\sskip
Let $I^\B_\lba\subset\B_{\bs\la}$ be the ideal generated by
the elements $\si_i(\bs z)-a_i$, $i=1,\dots,n$. Consider the subspace
$\,I^\M_\lba=I^\B_\lba\>\M_{\bs\la}=I^\V_{\bs a}\cap\M_{\bs\la}$,
where $I^\V_{\bs a}$ is given by~\Ref{IVa}. Recall the ideal
$I^\O_\lba$ defined in Section~\ref{wronski}.

\begin{lem}
\label{identify}
We have
\vvn.2>
\be
\tau_{\bs\la}({I^\O_\lba})={I^\B_\lba}\,,\qquad
\mu_{\bs\la}({I^\O_\lba})={I^\M_\lba}\,,\qquad
\B_\lba=\B_{\bs\la}/{I^\B_\lba}\,,\qquad
\M_\lba=\M_{\bs\la}/{I^\M_\lba}\,.
\vv.4>
\ee
\end{lem}
\begin{proof}
The lemma follows from Theorems~\ref{first}, \ref{first1} and
Lemmas~\ref{symm OK}, \ref{factor=weyl}.
\end{proof}

By Lemma~\ref{identify},
the maps $\tau_{\bs\la}$ and $\mu_{\bs\la}$ induce the maps
\beq
\label{taumu}
\tau_\lba:\O_\lba\to\B_\lba\,,\qquad
\mu_\lba:\O_\lba\to\M_\lba\,.
\eeq

\begin{thm}
\label{second}
The map $\tau_\lba$ is an isomorphism of algebras. The map $\mu_\lba$ is
an isomorphism of vector spaces. For any $F,G\in\O_\lba$,
we have
\vvn-.2>
\be
\mu_\lba(FG)\,=\,\tau_\lba(F)\,\mu_\lba(G)\,.
\vv.3>
\ee
In other words, the maps\/ $\tau_\lba$ and\/ $\mu_\lba$ give an isomorphism of
the regular representation of\/ $\O_\lba$ and the\/ $\B_\lba$-module $\M_\lba$.
\end{thm}
\begin{proof}
The theorem follows from Theorems~\ref{first}, \ref{first1}
and Lemma~\ref{identify}.
\end{proof}

\begin{rem}
By Lemma~\ref{local Wr}, the algebra $\O_\lba$ is Frobenius.
Therefore, its regular and coregular representations are isomorphic.
\end{rem}

\subsection{Isomorphism of algebras $\O_\Llb$ and $\B_\Llb$}
Let $\bs\La=(\bs\la^{(1)},\dots,\bs\la^{(k)})$ be a sequence of partitions
with at most $N$ parts such that $|\bs\la^{(s)}|=n_s$ for all $s=1,\dots,k$.

In this section we show that the Bethe algebra $\B_{\bs\La, \bs\la,\bs b}$\>,
associated with the space
$\M_\Llb=(\otimes_{s=1}^kL_{\bs\la^{(s)}}(b_s))_{\bs\la}^{sing}$,
see~\Ref{MMM}, is isomorphic to the algebra $\O_\Llb$, and that under this
isomorphism the $\B_{\bs\La, \bs\la,\bs b}$-module $\M_\Llb$ is isomorphic
to the coregular representation of the algebra $\O_\Llb$\>.

\sskip
We begin with a simple observation. Let $A$ be an associative unital algebra,
and let $L,M$ be $A$-modules such that $L$ is isomorphic to a subquotient of
$M$. Denote by $A_L$ and $A_M$ the images of $A$ in $\End(L)$ and $\End(M)$,
respectively, and by $\pi_L:A\to A_L$, $\pi_M:A\to A_M$ the corresponding
epimorphisms. Then, there exists a unique epimorphism $\pi_{ML}:A_M\to A_L$
such that $\pi_L=\pi_{ML}\circ\pi_M$.

Applying this observation to the Bethe algebra $\B$ and $\B$-modules
$\M_{\bs\la}\,$, $\M_\lba\,$, $\M_\Llb\,$,
we get a chain of epimorphisms $\B\to\B_{\bs\la}\to
\B_\lba\to\B_\Llb\,$. In particular, smaller spaces
are naturally modules over Bethe algebras associated to bigger spaces.

%For any element $F\in\B_{\bs\la}$, we denote by $\bat F$ the projection of
%$F$ to the algebra $\B_\lba$.

\sskip
Let $\>C_i(u)$, \,${i=1,\dots,N}$, be the Laurent series in $u^{-1}$ obtained
by projecting each coefficient of the series $B_i(u)\prod_{s=1}^n(u-z_s)$
to $\B_{\bs\la}$.

\begin{lem}
\label{capelli}
The series $C_i(u)$, $i=1,\dots,N$, $i\leq n$, are polynomials in\/
$u$ of degree $n-i$. For $i=n+1,n+2,\dots,N$, the series $C_i(u)$ is zero.
\end{lem}
\begin{proof}
Lemma~\ref{capelli} follows from Theorem~2.1 in~\cite{MTV3}.
Alternatively, Lemma~\ref{capelli} follows from the formula
$\>C_i(u)=\tau_{\bs\la} \bigl(G_i(u)\bigr)$.
\end{proof}

Let $\>\bat C_i(u)$, \,${i=1,\dots,N}$, be the polynomial
obtained by projecting each coefficient of the polynomial $C_i(u)$ to
$\B_\lba$.

Introduce the elements $\bat C_{ijs}\in\B_\lba$ for $i=1,\dots,N$,
$j=0,1,\dots,n-i$, $s=1,\dots,k$, by the rule
\be
\sum_{j=0}^{n-i}\,\bat C_{ijs}\,(u-b_s)^j\,=\,\bat C_i(u)\,.
\ee
In addition, let $\>\bat C_{0js}$\>, \,$j=0,\dots,n$, $s=1,\dots,k$,
be the numbers such that
\be
\sum_{j=0}^n\,\bat C_{0js}\,(u-b_s)^j\,=\,\prod_{r=1}^k\,(u-b_r)^{n_r}\,.
\ee

\begin{lem}
\label{nil2}
The elements $\bat C_{ijs}$, $\,i=1,\dots,N$, $\,s=1,\dots,k$,
$\,j=0,1,\dots,n_s-i-1$, are nilpotent.
\end{lem}
\begin{proof}
We clearly have $\tau_\lba\bat G_{ijs}=\bat C_{ijs}$.
Lemma~\ref{nil2} follows from Lemma~\ref{nil1}.
\end{proof}

Define the {\it indicial polynomial\/ $\bat\chi_s^\B(\al)$ at\/ $b_s$}
by the formula
\be
\bat\chi_s^\B(\al)\,=\,
\sum_{i=0}^{N}\,\bat C_{i,n_s-i,s}\prod_{j=0}^{N-i-1} (\al-j)\,.
\ee
It is a polynomial of degree $N$ in variable $\al$ with coefficients in
$\B_\lba$.

Let $\bat I^\B_\Llb$ be the ideal in $\B_\lba\>$ generated by the elements
$\>\bat C_{ijs}$, $\,i=1,\dots,N$, $\,s=1,\dots,k$, $j=0,1,\dots,n_s-i-1$,
and the coefficients of the polynomials
\beq
\label{chiB}
\bat\chi_s^\B(\al)\,-\,\prod_{\fratop{r=1}{r\ne s}}^k\,(b_s-b_r)^{n_r}\,
\prod_{l=1}^N\,(\al-\la_l^{(s)}-N+l)\,,\qquad s=1,\dots,k\,.\kern-3em
\eeq

\begin{lem}
\label{zero}
The ideal\/ $\bat I^\B_\Llb$ belongs to the kernel of the projection\/
$\>\B_\lba\to\B_\Llb\,$.
\end{lem}
\begin{proof}
Set $\bat C_0(u)=\prod_{s=1}^k\,(u-b_s)^{n_s}$.
\vvn.16>
A straightforward calculation shows that the action of the polynomial
$\sum_{i=0}^N\,\bat C_i(b_s)\prod_{j=0}^{N-i-1}(\al-j)$ on
$\M_\lba\subset \otimes_{r=1}^kL_{\bs\la^{(r)}}(b_r)$
coincides with the action of the operator
\vvn-.4>
\be
\prod_{\fratop{r=1}{r\ne s}}^k\,(b_s-b_r)^{n_r}\left(1^{\otimes(s-1)}
\otimes Z(\al-N+1)\otimes 1^{\otimes(k-s)}\right),
\vv-.3>
\ee
cf.~\cite{MTV2}.
The lemma follows from Theorem~\ref{Zcent} and formula~\Ref{Zxv}.
\end{proof}

Hence, the projection $\B_\lba\to\B_\Llb$ descends to an epimorphism
\beq
\label{pi}
\pi_\Llb:\B_\lba/\bat I^\B_\Llb\,\to\,\B_\Llb\,,
\eeq
which makes $\M_\Llb$ into a $\>\B_\lba/\bat I^\B_\Llb$-module.

\sskip
Denote $\,\ker\>(\bat I^\B_\Llb)=
\{\>v\in\M_\lba\ |\ \bat I^\B_\Llb\>v=0\>\}\,$.
Clearly, $\,\ker\>(\bat I^\B_\Llb)$ is a $\>\B_\lba$-submodule of $\M_\lba$.

\begin{prop}
\label{kerI}
The $\>\B_\lba/\bat I^\B_\Llb$-modules $\,\ker\>(\bat I^\B_\Llb)$
and\/ $\M_\Llb$ are isomorphic.
\end{prop}
\noindent
The proposition is proved in Section~\ref{main proof}.

\goodbreak
\medskip
Let $\bat I^\O_\Llb\subset\O_\lba$ be the ideal defined in Section~\ref{barI}.
\vvn.1>
Clearly, the map $\tau_\lba:\O_\lba\to\B_\lba$ sends $\bat I^\O_\Llb$ to
$\bat I^\B_\Llb$. By Lemma~\ref{scheme}, the maps $\tau_\lba$ and $\pi_\Llb$
induce the homomorphism
\vvn.1>
\be
\tau_\Llb:\O_\Llb\,\to\,\B_\Llb\,.
\vv.1>
\ee

By Theorem~\ref{second}, the map $\mu_\lba:\O_\lba\to\M_\lba$ sends
$\Ann(\bat I^\O_\Llb)\subset\O_\lba$ to $\ker\>(\bat I^\B_\Llb)$\>.
The vector spaces $\Ann(\bat I^\O_\Llb)$ and $(\O_\Llb)^*$ are isomorphic
by Corollary~\ref{AnnI}. Hence, Proposition~\ref{kerI} yields that the map
$\mu_\lba$ induces a bijective linear map
\be
\mu_\Llb:(\O_\Llb)^*\to\,\M_\Llb\,.
\vv.1>
\ee

\sskip
For any $F\in\O_\Llb$, denote by $F^*\in\End\bigl((\O_\Llb)^*\bigr)$
the operator, dual to the operator of multiplication by $F$ on $\O_\Llb$.

\begin{thm}
\label{third}
The map $\tau_\Llb$ is an isomorphism of algebras.
For any $F\in\O_\lba$ and $G\in(\O_\Llb)$\>, we have
\vvn-.2>
\be
\mu_\Llb(F^*G)\,=\,\tau_\Llb(F)\,\mu_\Llb(G)\,.
\vv.3>
\ee
In other words, the maps\/ $\tau_\Llb$ and\/ $\mu_\Llb$ give an isomorphism of
the coregular representation of\/ $\O_\Llb$ on the dual space $(\O_\Llb)^*$
and the $\B_\Llb$-module $\M_\Llb$.
\end{thm}
\begin{proof}
By Lemma~\ref{scheme}, the isomorphism $\tau_\lba:\O_\lba\to\B_\lba$ induces
the isomorphism
\be
\tau_\Llb:\O_\Llb\,\to\,\B_\lba/\bat I^\B_\Llb\,.
\ee
so the maps $\tau_\Llb$ and $\mu_\lba$ give an isomorphism
of the $\O_\Llb$-module $\Ann(\bat I^\O_\Llb)$
and the $\B_\lba/\bat I^\B_\Llb$-module
$\ker\>(\bat I^\B_\Llb)$, see Theorem~\ref{second}.

By Lemma~\ref{coreg}, the $\O_\Llb$-module $\Ann(\bat I^\O_\Llb)$ is
isomorphic to the coregular representation of $\O_\Llb$ on the dual space
$(\O_\Llb)^*$. In particular, it is faithful.
Therefore, the $\B_\lba/\bat I^\B_\Llb$-module $\ker\>(\bat I^\B_\Llb)$
is faithful.
By Proposition~\ref{kerI}, the $\>\B_\lba/\bat I^\B_\Llb$-module
$M_{\bs\la,\bs\la,\bs b}\,$, isomorphic to $\ker\>(\bat I^\B_\Llb)$,
is faithful
too, which implies that the map $\pi_\Llb:\B_\lba/\bat I^\B_\Llb\to\B_\Llb$
is an isomorphism of algebras. The theorem follows.
\end{proof}

\begin{rem}
By Lemma~\ref{frobenius}, the algebra $\O_\Llb$ is Frobenius.
Therefore, its coregular and regular representations are isomorphic.
\end{rem}

\subsection{Proof of Proposition~\ref{kerI}}
\label{main proof}
We begin the proof with an elementary auxiliary lemma. Let $M$ be
a finite-dimensional vector space, $U\subset M$ a subspace, and $E\in\End(M)$.
\begin{lem}
\label{EU}
Let\/ $EM\subset U$, and the restriction of\/ $E$ to\/ $U$ is invertible in\/
$\End(U)$. Then\/ $EU=U$ and $\,M=U\oplus\ker E$.
\qed
\end{lem}

Let $W_m$ be the $\gln[t]$-module defined in Section~\ref{secweyl},
and $\bs\mu$ a partition
with at most $N$ parts such that $|\bs\mu|=m$. Recall that $W_m$ is a graded
vector space, and the grading of $W_m$ is defined in Lemma~\ref{weyl}.

Given a homogeneous vector $w\in(W_m)_{\bs\mu}^{sing}$, let $\L_w(b)$ be the
$\glnt$-submodule of $W_m(b)$ generated by the vector $v$. The space $\L_w(b)$
is graded. Denote by $\L_w^=(b)$ and $\L_w^>(b)$ the subspaces of $\L_w(b)$
spanned by homogeneous vectors of degree $\>\deg w$ and of degree strictly
greater than $\>\deg w$, respectively. The subspace $\L_w^=(b)$ is
a $\gln$-submodule of $\L_w(b)$ isomorphic to the irreducible $\gln$-module
$L_{\bs\mu}$. The subspace $\L_w^>(b)$ is a $\glnt$-submodule of $\L_w(b)$,
and the $\glnt$-module $\L_w(b)/\L_w^>(b)$ is isomorphic to the evaluation
module $L_{\bs\mu}(b)$. If $v$ has the largest degree among vectors
in $(W_m)_{\bs\mu}^{sing}$, then $\L_w^>(b)$, considered as a $\gln$-module,
does not contain $L_{\bs\mu}$.

For any $s=1,\dots,k$, pick up a homogeneous vector
$w_s\in(W_{n_s})_{\bs\la^{(s)}}^{sing}$ of the largest possible degree.
Let $\L_{\bs w}(\bs b)$ be the $\glnt$-submodule of
$\otimes_{s=1}^k W_{n_s}(b_s)$ generated by the vector $\otimes_{s=1}^k w_s$.

Denote by $\L_{\bs w}^=(\bs b)$ and $\L_{\bs w}^>(\bs b)$
the following subspaces of $\L_{\bs w}(\bs b)$:
\begin{gather*}
\L_{\bs w}^=(\bs b)\,=\,\otimes_{s=1}^k\L_{w_s}^=(b_s)\,,
\\
\L_{\bs w}^>(\bs b)\,=\,\sum_{s=1}^k\,\L_{w_1}(b_1)\otimes\dots\otimes
\L_{w_s}^>(b_s)\otimes\dots\otimes\L_{w_k}(b_k)\,.
\end{gather*}
The subspace $\L_{\bs w}^>(\bs b)$ is a $\glnt$-submodule
of $\L_{\bs w}(\bs b)$, and the $\glnt$-module
$\L_{\bs w}(\bs b)/\L_{\bs w}^>(\bs b)$ is isomorphic to the tensor
product of evaluation modules $\otimes_{s=1}^kL_{\bs\la^{(s)}}(b_s)$.

The space $\otimes_{s=1}^k W_{n_s}$ has the second $\glnt$-module structure,
denoted $\gr\bigl(\otimes_{s=1}^k W_{n_s}(b_s)\bigr)$, which was introduced
at the end of Section~\ref{secweyl}. The subspace $\L_{\bs w}(\bs b)$ is
a $\glnt$-submodule of $\gr\bigl(\otimes_{s=1}^k W_{n_s}(b_s)\bigr)$,
isomorphic to a direct sum of irreducible $\glnt$-modules of the form
$\otimes_{s=1}^k L_{\bs\mu^{(s)}}(b_s)$, where $|\bs\mu^{(s)}|=n_s$,
$\,s=1,\dots,k$, see Lemmas~\ref{grWmb} and~\ref{grtensor}, and
$(\bs\mu^{(1)},\dots,\bs\mu^{(k)})\ne(\bs\la^{(1)},\dots,\bs\la^{(k)})$
for any term of the sum.

The subspace $\M_\lba=(\otimes_{s=1}^k W_{n_s}(b_s))_{\bs\la}^{sing}$
is invariant under the action of the Bethe algebra ${\B\subset\Uglnt}$.
This makes it into a $\B$-module, which we call the standard $\B$-module
structure on $\M_\lba$. The $\B$-module $\M_\lba$ contains
the submodules
$\M_\Llb^{\bs w}=(\L_{\bs w}(\bs b))_{\bs\la}^{sing}$
and $\M_\Llb^{\bs w,>}=(\L_{\bs w}^>(\bs b))_{\bs\la}^{sing}$,
and the subspace
$\M_\Llb^{\bs w,=}=(\L_{\bs w}^=(\bs b))_{\bs\la}^{sing}$.
As vector spaces,
$\M_\Llb^{\bs w}=\M_\Llb^{\bs w,=}
\oplus\M_\Llb^{\bs w,>}\,$. The $\B$-modules
$\M_\Llb^{\bs w}/\M_\Llb^{\bs w,>}$
and $\M_\Llb$ are isomorphic.

The space $\M_\lba$ has another $\B$-module structure,
inherited from the $\glnt$-module structure of
$\gr\bigl(\otimes_{s=1}^k W_{n_s}(b_s)\bigr)$.
We denote the new structure $\gr\M_\lba\,$.
The subspaces $\M_\Llb^{\bs w}\,$,
$\M_\Llb^{\bs w,=}\,$, $\M_\Llb^{\bs w,>}$
are $\B$-submodules of the $\B$-module $\gr\M_\lba\,$. The submodule
$\M_\Llb^{\bs w,=}\subset\gr\M_\lba$ is isomorphic
to the $\B$-module $\M_\Llb\,$, and the submodule
$\M_\Llb^{\bs w,>}\subset\gr\M_\lba$ is isomorphic
to a direct sum of $\B$-modules of the form $\M_{\bs\Mu,\bs\la,\bs b}\,$, where
$\bs\Mu=(\bs\mu^{(1)},\dots,\bs\mu^{(k)})$, \,$|\bs\mu^{(s)}|=n_s$,
$\,s=1,\dots,k$, and $\bs\Mu\ne\bs\La$ for any term of the sum.

\sskip
In the picture described above, we can regard all $\B$-modules involved
as $\B_\lba$-modules.

\sskip
For any $F\in\B_\lba$, we denote by $\gr F\in\End(\M_\lba)$ the linear operator
corresponding to the action of $F$ on $\gr\M_\lba$\>. The map $F\mapsto\gr F$
is an algebra homomorphism.

Let complex numbers $c_1,\dots,c_k$, $\al_1,\dots,\al_k$ be such that
\be
\sum_{s=1}^k\,c_s\>\Bigl(\,\prod_{i=1}^N\,(\al_s-\mu_i^{(s)}-N+i)-
\prod_{i=1}^N\,(\al_s-\la_i^{(s)}-N+i)\Bigr)\,\ne\,0\,,
\ee
for any sequence of partitions
$(\bs\mu^{(1)},\dots,\bs\mu^{(k)})\ne(\bs\la^{(1)},\dots,\bs\la^{(k)})\,$. Let
\vvn.1>
\be
E\,=\,\sum_{s=1}^k c_s\>\Bigl(\bat\chi_s^\B(\al_s)-
\prod_{\fratop{r=1}{r\ne s}}^k\,(b_s-b_r)^{n_r}\,
\prod_{i=1}^N\,(\al_s-\la_i^{(s)}-N+i)\Bigr)\,\in \bat I^\B_\Llb\,,
\vv-.2>
\ee
where $\bat\chi_s^\B$ is the indicial polynomial~\Ref{chiB}.
\vvn.1>
With respect to the standard $\B$-module structure on $\M_\lba$,
we have $E\>\M_\Llb^{\bs w}\subset\M_\Llb^{\bs w,>}$.

\begin{lem}
The restriction of\/ $E$ to $\M_\Llb^{\bs w,>}$ is invertible in
$\End(\M_\Llb^{\bs w,>})$.
\end{lem}
\begin{proof}
Lemma~\ref{zero} implies that the projection of $E$ to $\B_\Llb$ equals zero,
and the projection of $E$ to $\B_{\bs\Mu,\bs\la,\bs b}$ with $\bs\Mu\ne\bs\La$
is invertible. This means that the restriction of the operator $\gr E$
to $\M_\Llb^{\bs w,>}$ is invertible in $\End(\M_\Llb^{\bs w,>})$.
Therefore, the restriction of\/ $E$ to $\M_\Llb^{\bs w,>}$ is invertible
in $\End(\M_\Llb^{\bs w,>})$.
\end{proof}

Denote $\ker_\Llb^{\bs w}E\>=\>\ker E\,\cap\,\M_\Llb^{\bs w}\,$.
By Lemma~\ref{EU}, the canonical projection
\be
\M_\Llb^{\bs w}\to
\M_\Llb^{\bs w}/\M_\Llb^{\bs w,>}\simeq
\M_\Llb
\ee
induces an isomorphism $\ker_\Llb^{\bs w}E\to\M_\Llb$ of vector spaces.
\vvn.1>
Since the algebra $\B_\lba$ is commutative, the subspace $\ker_\Llb^{\bs w}E$
\vvn.1>
is a $\B$-submodule, and the map $\ker_\Llb^{\bs w}E\to\M_\Llb$ is
an isomorphism of $\B_\lba$-modules.

Lemma~\ref{zero} implies that elements of the ideal $\bat I^\B_\Llb$ act on
$\M_\Llb$ by zero. Hence, they act by zero on $\ker_\Llb^{\bs w}E$, that is,
$\ker_\Llb^{\bs w}E\subset \ker\>(I^\B_\Llb)\,$. On the other hand, we have
\be
\dim\>\ker\>(\bat I^\B_\Llb)\,=\,\dim\>\Ann(\bat I^\O_\Llb)\,=\,
\dim\>\O_\Llb\,=\,\dim\>\M_\Llb\,=\,\dim\>\ker_\Llb^{\bs w}E\,,
\ee
see Theorem~\ref{second}, Corollary~\ref{AnnI} and formula~\Ref{dimO},
which yields \,$\ker_\Llb^{\bs w}E=\ker\>(\bat I^\B_\Llb)\,$.
Proposition~\ref{kerI} is proved.
\qed

\begin{rem}
Note that formula~\Ref{dimO} is an important ingredient of the proof.
\end{rem}

\section{Applications}
\label{app sec}
In this section we state some of the obtained results in a way, relatively
independent from the main part of the paper.
For convenience, we recall some definitions and facts.

The Bethe algebra $\B$ is a commutative subalgebra of $\Uglnt$, defined
in Section~\ref{bethesec}. It is generated by the elements $B_{ij}$,
$i=1,\dots,N$, $j\in\Z_{\ge i}$, which are coefficients of the series
$B_i(u)$, $i=1,\dots,N$, see formulae~\Ref{DB}, \Ref{Bi}.

\sskip
If $M$ is a $\B$-module and $\xi:\B\to\C$ a homomorphism, then
the eigenspace of $\B$-action on $M$ corresponding to $\,\xi\,$ is defined
as $\,\bigcap_{B\in\B}\ker(B|_M-\xi(B))$ and the generalized eigenspace of
$\B$-action on $M$ corresponding to $\,\xi\,$ is defined as
$\,\bigcap_{B\in\B}\bigl(\,\bigcup_{m=1}^\infty\ker(B|_M-\xi(B))^m\bigr)$.

\subsection{Tensor products of irreducible modules}
For a partition $\bs\la$ with at most $N$ parts, $L_{\bs\la}$ is
the irreducible finite-dimensional $\gln$-module of highest weight $\bs\la$.

Let $\bs\la^{(1)},\dots,\bs\la^{(k)}$ be partitions with at most $N$ parts,
$b_1,\dots,b_k$ distinct complex numbers. We are interested in the action
of the Bethe algebra $\B$ on the tensor product
$\otimes_{s=1}^kL_{\bs\la^{(s)}}(b_s)$ of evaluation $\glnt$-modules.

Since $\B$ commutes with the subalgebra $\Ugln\subset\Uglnt$,
the action of $\B$ preserves the subspace of singular vectors
$(\otimes_{s=1}^kL_{\bs\la^{(s)}}(b_s))^{sing}$ as well as
the weight subspaces of $\otimes_{s=1}^kL_{\bs\la^{(s)}}(b_s)$.
The action of $\B$ on $\otimes_{s=1}^kL_{\bs\la^{(s)}}(b_s)$ is determined
by the action of $\B$ on $(\otimes_{s=1}^kL_{\bs\la^{(s)}}(b_s))^{sing}$.

\sskip
Denote $\bs\La=(\bs\la^{(1)},\dots,\bs\la^{(k)})$. Given a partition $\bs\la$
with at most $N$ parts such that $|\bs\la|=\sum_{s=1}^k|\bs\la^{(s)}|\,$,
let $\Dl_\Llb$ be the set of all monic Fuchsian differential operators of
order $N$,
\beq
\label{mcD}
\D\,=\,\der^N+\>\sum_{i=1}^N\,h_i^\D(u)\,\der^{N-i}\,,
\eeq
where $\der=d/du$, with the following properties.

\begin{enumerate}
\flati
\item[\=a]
The singular points of $\D$ are at\/ $b_1,\dots,b_k$ and $\infty$ only.
\item[\=b]
The exponents of $\D$ at\/ \>$b_s$\>, \,$s=1,\dots,k$, are equal to
$\,\la_N^{(s)},\,\la_{N-1}^{(s)}+1,\alb\,\dots\,,\la_1^{(s)}+N-1\,$.
\item[\=c]
The exponents of $\D$ at $\infty$ are equal to
$\>1-N-\la_1,\,2-N-\la_2,\,\dots\,,-\>\la_N\,$.
\item[\=d]
The operator $\D$ is monodromy free.
\end{enumerate}
Equivalently, the last property can be replaced by 
\begin{enumerate}
\flati
\item[\=e]
The kernel of the operator $\D$ consists of polynomials only.
\end{enumerate}

A differential operator $\D$ belongs to the set $\Dl_\Llb$ if and only if
the kernel of $\D$ is a point of the intersection of Schubert cells
$\,\Om_\Lbl$\>, see Lemma~\ref{lem on intersection}.

Denote $n_s=|\bs\la^{(s)}|\,$, $\,s=1,\dots,k$, \,and $\,n=\sum_{s=1}^kn_s\,$.

\begin{thm}
\label{BL}
The action of the Bethe algebra $\,\B$ on
$(\otimes_{s=1}^kL_{\bs\la^{(s)}}(b_s))^{sing}$ has the following properties.
\begin{enumerate}
\item[(i)]
For every $i=1,\dots,N$, $i\leq n$, 
the action of the series\/ $B_i(u)$ is given by
\vvn.1>
the power series expansion in\/ $u^{-1}\!$ of a rational function of the form
\vvn.2>
$\,A_i(u)\prod_{s=1}^k(u-b_s)^{-n_s}$,
where $A_i(u)$ is a polynomial of degree $\,n-i$ with coefficients in
$\,\End\bigl((\otimes_{s=1}^kL_{\bs\la^{(s)}})^{sing}\bigr)$. 
For $i=n+1, n+2,\dots, N$, the series $B_i(u)$ acts by zero.

\sskip
\item[(ii)]
The image of $\,\B$ in
$\,\End\bigl((\otimes_{s=1}^kL_{\bs\la^{(s)}})^{sing}\bigr)$
\vvn.1>
is a maximal commutative subalgebra of dimension
$\,\dim\>(\otimes_{s=1}^kL_{\bs\la^{(s)}})^{sing}$.

\sskip
\item[(iii)]
Each eigenspace of the action of $\,\B$ is one-dimensional.

\sskip
\item[(iv)] Each generalized eigenspace of the action of $\,\B$ is
generated over $\,\B$ by one vector.

\sskip
\item[(v)] The eigenspaces of the action of $\,\B$ on
$(\otimes_{s=1}^kL_{\bs\la^{(s)}}(b_s))_{\bs\la}^{sing}$ are in a one-to-one
correspondence with differential operators from $\>\Dl_\Llb\,$. Moreover,
if\/ $\D$ is the differential operator, corresponding to an eigenspace, then
the coefficients of the series\/ $h_i^\D(u)$ are the eigenvalues of the action
of the respective coefficients of the series\/ $\,B_i(u)$.

\sskip
\item[(vi)] The eigenspaces of the action of $\,\B$ on
$(\otimes_{s=1}^kL_{\bs\la^{(s)}}(b_s))_{\bs\la}^{sing}$ are in a one-to-one
correspondence with points of the intersection of Schubert cells
$\,\Om_\Lbl$\>, given by~\Ref{Omega}.
\end{enumerate}
\end{thm}
\begin{proof}
The first property follows from Lemma~\ref{capelli}. The other properties
follow from Theorem~\ref{third}, Lemma~\ref{lem on intersection}, and standard
facts about the coregular representations of Frobenius algebras given in
Section~\ref{comalg}.
\end{proof}

The intersection of Schubert cells $\Om_\Lbl$ is {\it transversal\/} if
the scheme-theoretic intersection $\O_\Llb$ is a direct sum of one-dimensional
algebras.

\begin{cor}
\label{card}
The following statements are equivalent.
\begin{enumerate}
\item[(i)]
The action of the Bethe algebra $\>\B$ on
$(\otimes_{s=1}^kL_{\bs\la^{(s)}}(b_s))_{\bs\la}^{sing}$ is diagonalizable.

\sskip
\item[(ii)]
The set $\>\Dl_\Llb\,$ consists of\/
$\,\dim\>(\otimes_{s=1}^kL_{\bs\la^{(s)}})_{\bs\la}^{sing}$ distinct points.

\sskip
\item[(iii)]
The set $\,\Om_\Lbl$ consists of\/
$\,\dim\>(\otimes_{s=1}^kL_{\bs\la^{(s)}})_{\bs\la}^{sing}$ distinct points.

\sskip
\item[(iv)]
The intersection of Schubert cells $\,\Om_\Lbl$ is transversal.
\qed
\end{enumerate}
\end{cor}

We call a differential operator
$\D\in\Dl_\Llb$ {\it real} if all $h_i^\D(u)$ are rational
functions with real coefficients. We call $X\in \Om_\Lbl$ {\it real}
if $X$ has a basis consisting of polynomials with real coefficients.

\goodbreak
\begin{cor}
\label{real}
Let $b_1,\dots,b_k$ be distinct real numbers.
\begin{enumerate}
\item[(i)] The set $\>\Dl_\Llb\,$ consists of\/
$\,\dim\>(\otimes_{s=1}^kL_{\bs\la^{(s)}})_{\bs\la}^{sing}$
distinct real points.
\item[(ii)]
The intersection of Schubert cells $\,\Om_\Lbl$ consists of\/
$\dim\>(\otimes_{s=1}^kL_{\bs\la^{(s)}})_{\bs\la}^{sing}$ distinct real points
and is transversal.
\end{enumerate}
\end{cor}
\begin{proof}
Denote by $\B^\R$ and $L_{\bs\la}^\R$ the real forms of the Bethe algebra $\B$
and the module $L_{\bs\la}$, respectively. The real vector space
$(\otimes_{s=1}^kL_{\bs\la^{(s)}}^\R)_{\bs\la}^{sing}$ has a natural positive
definite bilinear form, which comes from the tensor product of the Shapovalov
forms on the tensor factors. If $b_1,\dots,b_k$ are distinct real numbers, and
$B\in\B^\R$, then $B$ acts on
$(\otimes_{s=1}^kL_{\bs\la^{(s)}}^\R(b_s))_{\bs\la}^{sing}$ as a linear
operator, symmetric with respect to that form, see~\cite{MTV1}, \cite{MTV2}.
In particular $B$ is diagonalizable and all its eigenvalues are real.
The corollary follows.
%
%Denote by $\B^\R$ and $L_{\bs\la}^\R$ the natural real counterparts of
%the Bethe algebra $\B$ and the module $L_{\bs\la}$, respectively. The space
%$(\otimes_{s=1}^kL_{\bs\la^{(s)}}^\R)_{\bs\la}^{sing}$ has a natural real
%positive definite bilinear form, called Shapovalov form. If $b_1,\dots,b_k$
%are distinct real numbers, and $B\in\B^\R$, then $B$ acts
%$(\otimes_{s=1}^kL_{\bs\la^{(s)}}^\R(b_s))_{\bs\la}^{sing}$ as a real linear
%operator, symmetric with respect to the Shapovalov form, see~\cite{MTV1},
%\cite{MTV2}. In particular $B$ is diagonalizable and all eigenvalues are real.
%
%The corollary follows.
\end{proof}

For $N=2$, this corollary is obtained in~\cite{EG}.

\subsection{Weyl modules}
Let $n_1,\dots, n_k$ be natural numbers and $b_1,\dots,b_k$ distinct complex
numbers. Let $\otimes_{s=1}^k W_{n_s}(b_s)$ be the Weyl module associated to
$\bs n,\bs b$ defined in Section~\ref{secweyl}.
We are interested in the action of the Bethe algebra $\B$ on
the Weyl module $\otimes_{s=1}^k W_{n_s}(b_s)$.

The action of $\B$ preserves the subspace of singular
vectors $(\otimes_{s=1}^k W_{n_s}(b_s))^{sing}$ as well as the weight subspaces
of $\otimes_{s=1}^k W_{n_s}(b_s)$. The action of $\B$ on
$\otimes_{s=1}^k W_{n_s}(b_s)$ is determined by the action of $\B$
on $(\otimes_{s=1}^k W_{n_s}(b_s))^{sing}$.

\sskip
Recall that $V$ denotes the irreducible $\gln$-module of highest weight
$(1,0,\dots,0)$, which is the vector representation of $\gln$.
Set $n=n_1+\dots+n_k$.

\sskip
Denote by $\Dlnb$ the set of all monic differential
operators $\D$ of order $N$ with the following properties.

\begin{enumerate}
\flati
\item[\=a]
The kernel of the operator $\D$ consists of polynomials only.
\item[\=b]
The first coefficient $h_1^\D(u)$ of $\D$, see~\Ref{mcD},
is equal to $\>-\sum_{s=1}^k n_s\>(u-b_s)^{-1}$.
\end{enumerate}

\sskip\noindent
If $\D\in\Dlnb$, then $\D$ is a Fuchsian differential operator
with singular points at\/ $b_1,\dots,b_k$ and $\infty$ only.

\sskip
Denote by $\>\Wrnbi$ the set of all $N$-dimensional subspaces of $\C[u]$,
admitting a basis $p_1(u),\alb\dots,\alb p_N(u)$ with the Wronskian
\>$\Wr(p_1(u),\dots,p_N(u))=\prod_{s=1}^k(u-b_s)^{n_s}$.
\vvn.1>
A differential operator $\D$ belongs to the set $\Dlnb$ if and only if
the kernel of $\D$ belongs to the set $\>\Wrnbi$.

\begin{thm}
\label{BW}
The action of the Bethe algebra $\,\B$ on
$(\otimes_{s=1}^k W_{n_s}(b_s))^{sing}$ has the following properties.
\begin{enumerate}
\item[(i)]
For every $i=1,\dots,N$, $i\leq n$, 
the action of the series\/ $B_i(u)$ is given by
\vvn.1>
the power series expansion in\/ $u^{-1}\!$ of a rational function of the form
\vvn.2>
$\,A_i(u)\prod_{s=1}^k(u-b_s)^{-n_s}$,
where $A_i(u)$ is a polynomial of degree $\,n-i$ with coefficients in
$\,\End\bigl((\otimes_{s=1}^k W_{n_s})^{sing}\bigr)$. 
For $i=n+1,n+2,\dots,N$, the series $B_i(u)$ acts by zero.

\sskip
\item[(ii)]
The image of $\,\B$ in $\,\End\bigl((\otimes_{s=1}^k W_{n_s})^{sing}\bigr)$
\vvn.1>
is a maximal commutative subalgebra of dimension
$\,\dim\>(V^{\otimes n})^{sing}$.

\sskip
\item[(iii)]
Each eigenspace of the action of $\,\B$ is one-dimensional.

\sskip
\item[(iv)] Each generalized eigenspace of the action of $\,\B$ is
generated over $\,\B$ by one vector.

\sskip
\item[(v)] The eigenspaces of the action of $\,\B$ on
$(\otimes_{s=1}^k W_{n_s}(b_s))^{sing}$ are in a one-to-one correspondence
with differential operators from $\>\Dlnb\,$. Moreover,
if\/ $\D$ is the differential operator, corresponding to an eigenspace, then
the coefficients of the series\/ $h_i^\D(u)$ are the eigenvalues of the action
of the respective coefficients of the series\/ $\,B_i(u)$.

\sskip
\item[(vi)] The eigenspaces of the action of $\,\B$ on
\vvn.1>
$(\otimes_{s=1}^k W_{n_s}(b_s))^{sing}$ are in a one-to-one
correspondence with spaces of polynomials from\/ $\Wrnbi$.
\end{enumerate}
\end{thm}
\begin{proof}
The first property follows from Lemmas~\ref{factor=weyl} and~\ref{capelli}.
The other properties follow from Theorem~\ref{second}, formulae~\Ref{DOla}
and~\Ref{F1}, and standard facts about the regular representations of Frobenius
algebras given in Section~\ref{comalg}.
\end{proof}

\begin{cor}
\label{cardW}
The following three statements are equivalent.
\begin{enumerate}
\item[(i)]
The action of the Bethe algebra $\>\B$ on
$(\otimes_{s=1}^k W_{n_s}(b_s))^{sing}$ is diagonalizable.

\sskip
\item[(ii)]
The set $\>\Dlnb\,$ consists of\/
$\,\dim\>(V^{\otimes n})^{sing}$ distinct points.

\sskip
\item[(iii)]
The set $\,\Wrnbi$ consists of\/
$\,\dim\>(V^{\otimes n})^{sing}$ distinct points.
\qed
\end{enumerate}
\end{cor}

\begin{rem}
It is easy to see that the set $\Dlnb$ is a disjoint union of
the sets $\Dl_\Llb$ with $|\bs\la^{(s)}|=n_s$, \,$s=1,\dots,k$. Similarly,
the set $\Wrnbi$ is a disjoint union of the sets $\Om_\Lbl$
with $|\bs\la^{(s)}|=n_s$, \,$s=1,\dots,k$. Comparing items (ii), (iii)
of Corollaries~\ref{cardW} and~\ref{card}, one can see that the sets
$\Dlnb$ and $\Wrnbi$ can have cardinality $\dim\>(V^{\otimes n})$
only if for each $s=1,\dots,k$ the decomposition of the $\gln$-module
$V^{\otimes n_s}$ into the direct sum of irreducible modules is multiplicity
free, that is, each of the numbers $n_1,\dots,n_k$ equals $1$ or $2$.
\end{rem}

\subsection{Monodromy free Fuchsian differential operators}
In this section we give a
reformulation of a part of Corollary~\ref{real}.

Let $\bs\mu^{(s)}=(\mu_1^{(s)},\dots,\mu_N^{(s)})$, \,$s=0,\dots,k$, be
sequences of nonincreasing integers, $\mu_1^{(s)}\ge\dots\ge\mu_N^{(s)}$, and
$\bs b=(b_0,\dots,b_k)$ a sequence of distinct points on the Riemann sphere.
Set $\bs\Mu=(\bs\mu^{(0)},\dots,\bs\mu^{(k)})$.

\sskip
Denote by $\Dl_\Mmb$ the set of all monic Fuchsian differential operators
of order $N$,
\beq
\label{mcd2}
\D\,=\,\der^N+\>\sum_{i=1}^N\,h_i^\D(u)\,\der^{N-i}\,,
\eeq
where $\der=d/du$, with the following properties.

\begin{enumerate}
\flati
\item[\=a]
The singular points of $\D$ are at\/ $b_0,\dots,b_k$ only.
\item[\=b]
The exponents of $\D$ at\/ \>$b_s$\>, \,$s=0,\dots,k$, are equal to
\vv.2>
$\,\mu_N^{(s)},\,\mu_{N-1}^{(s)}+1,\alb\,\dots\,,\mu_1^{(s)}+N-1\,$.
\item[\=c]
The operator $\D$ is monodromy free.
\end{enumerate}
If $\D\in \Dl_\Mmb$ then the kernel of $\D$ consists of rational
functions with poles at $b_0,\dots,b_k$ only.

\sskip
For a sequence $\bs\mu=(\mu_1,\dots,\mu_N)$ of nonincreasing integers, let
$L_{\bs\mu}$ be the irreducible finite-dimensional $\gln$-module of highest
weight $\bs\mu$. Given a $\gln$-module $M$, denote by $M^{\gln}$ the subspace
of $\gln$-invariants in $M$. The subspace $M^{\gln}$ is the multiplicity space
of the trivial $\gln$-module $L_{(0,\dots,0)}$ in $M$.

\begin{thm}
\label{circle}
Let\/ \>$b_0,\dots,b_k$ be distinct points on the Riemann sphere which lie
on a circle or on a line. Then the set\/ $\Dl_\Mmb$ consists of\/
\>$\dim(\otimes_{s=0}^k L_{\bs\mu^{(s)}})^{\gln}$ distinct points.
\end{thm}
\begin{proof}
We reduce the statement to item (i) of Corollary~\ref{real}.

Making if necessary a fractional-linear change of variable $u$ in the
differential operator $\D$, we can assume without loss of generality that
$b_0=\infty$ and the points $b_1,\dots,b_k$ are on the real axis.

\sskip
Let $c_0,\dots,c_k$ \>be integers such that $\sum_{s=0}^k c_s=0$.
For $s=0,\dots,k$, set
\vvn.3>
\beq
\label{mut}
\mut^{(s)}=\,(c_s+\mu_1^{(s)},\dots,c_s+\mu_N^{(s)}).
\vv.3>
\eeq
It is known from representation theory of $\gln$ that the vector spaces
$(\otimes_{s=0}^k L_{\bs\mu^{(s)}})^{\gln}$ and
$(\otimes_{s=0}^k L_{\bs\mut^{(s)}})^{\gln}$ are canonically isomorphic.
At the same time, we have a bijection $\Dl_\Mmb\to \Dl_{\bs{\tilde\Mu},\bs b}$
given by the conjugation
\vvn-.2>
\beq
\label{conjD}
\D\,\mapsto\,\prod_{s=1}^k\,(u-b_s)^{-c_s}\cdot\D\cdot
\prod_{s=1}^k\,(u-b_s)^{c_s}.
\eeq

Making transformations~\Ref{mut} and~\Ref{conjD} with $c_s=\mu_N^{(s)}$,
\,$s=1,\dots,k$, we can assume that $\mu_N^{(s)}=0$, $s=1,\dots,k$.
Then $\bs\mu^{(1)},\dots,\bs\mu^{(k)}$ are partitions with at most $N$ parts
and the kernel of every operator
$\D\in \Dl_\Mmb$ consists of polynomials only.

If $\mu_1^{(0)}\neq 0$ then the set $\Dl_\Mmb$ is empty and
\,$\dim\>(\otimes_{s=0}^k L_{\bs\mu^{(s)}})^{\gln}=0$.
Therefore, the theorem is trivially true.

If $\mu_1^{(0)}= 0$, then
\vvn-.6>
\be
\bs\la\,=\,(-\mu_N^{(0)},\dots,-\mu_1^{(0)})
\vv.3>
\ee
is a partition with at most $N$ parts and the space
$(\otimes_{s=0}^k L_{\bs\mu^{(s)}})^{\gln}$ is canonically isomorphic to the
space $(\otimes_{s=1}^k L_{\bs\mu^{(s)}})_{\bs\la}^{sing}$. The set $\Dl_\Mmb$
coincides with the set $\Dl_\Llb$ defined in the previous section with
$\bs\La=(\bs\mu^{(1)},\dots,\bs\mu^{(k)})$, and the statement of the theorem
follows from item (i) of Corollary~\ref{real}.
\end{proof}

\subsection{Monodromy of eigenspaces of the Bethe algebra}
\label{monodromy sec}
Let $\Theta \subset \C^n$ and $\Xi\subset \Om_{\lab}(\infty)$
be the Zariski open subsets described in
Lemma~\ref{generic}.
Consider the map $p : \C^n\to\C^n$, $\bs b =(b_1,\dots,b_n) \mapsto
\bs a =(a_1,\dots,a_n)$, where $a_i$ is the $i$-th elementary symmetric
function of $b_1,\dots,b_n$.
Then $\Pi= p(\Theta) $ is a
Zariski open subset of $\C^n$.

As vector spaces, $\M_\lba$ are identified for all $\bs a$ and
$\dim \M_\lba=\dim (V^{\otimes n})^{sing}_{\bs\la}$. By Lemma
\ref{generic}, for any $\bs a\in \Pi$, the vector space $\M_\lba$
has a basis consisting of eigenvectors of the Bethe algebra
$\B_\lba$ and distinct eigenvectors have different eigenvalues.
Therefore, all eigenspaces of $\B_\lba$ are one-dimensional,
and the eigenspaces depend on $\bs a$ holomorphically.

Let $\gamma$ be a loop in $\Pi$ starting at $\bs a$. Analytically continuing
along the loop we obtain a permutation of the set of the eigenspaces of
$\B_\lba$. This construction defines a homomorphism of the
fundamental group $\pi_1(\Pi,\bs a)$ to the symmetric group
$S_{\dim \M_{\bs\la,\bs a}}$. The image of the homomorphism will be called
{\it the monodromy group of eigenspaces of the Bethe algebra}.

\begin{prop}
\label{lem monodr}
The monodromy group acts transitively on the set of eigenspaces
of the Bethe algebra $\B_\lba$.
\end{prop}

\begin{proof}
Consider the set $\Upsilon$ consisting of all pair $(\bs a, \ell)$, where
$\bs a\in\Pi$ and $\ell$ is an eigenspace of the Bethe algebra acting on
$\M_\lba$. Define the map $w:\Upsilon\to\Pi,\ (\bs a,\ell)\mapsto\bs a$.

By Lemma~\ref{generic}, we have a holomorphic bijection
$\iota : \Upsilon \to \Xi$
which sends $(\bs a,\ell)$ to $X$ such that $\D_\ell = \D_X$.
We have $w = \Wr_{\bs\la}\circ \iota$\,, where $\Wr_{\bs\la}$ is the
Wronski map defined by~\Ref{wronski map}.

Let $(\bs a,\ell_1), (\bs a, \ell_2)$ be two eigenspaces of the Bethe algebra
acting on $\M_\lba$.
The subset $\Xi = \Wr_{\bs\la}^{-1}(\Pi)$ is a Zariski open
subset of an affine cell $\Omega_{\lab}$. In particular, it is path connected.
Connect the points $i(\bs a, \ell_1)$,
$ i(\bs a, \ell_2)$ by a curve $\delta$ in $\Xi$. Then
$\gamma=\Wr_{\bs\la}(\delta)$ is a loop in $\Pi$
whose monodromy sends $(\bs a,
\ell_1)$ to $ (\bs a, \ell_2)$.
\end{proof}

Consider two Bethe lines $(\bs a,\ell), \,(\bs a,\ell_2)$ in $\M_\lba$,
the differential operators $\D_{\ell_1},\,\D_{\ell_2}$ associated with
the lines, and the loop $\gamma$ constructed in the proof of Lemma
\ref{lem monodr}. The differential operators holomorphically depend on
$\bs a$.

\begin{cor}
Analytic continuation of $\D_{\ell_1}$ along $\gamma$ equals $\D_{\ell_2}$.
\qed
\end{cor}


\begin{thebibliography}{[AAAA]}
\normalsize
\frenchspacing
\raggedbottom

\bi [CG]{CG} V.\,Chari, J.\,Greenstein, {\it Current algebras,
highest weight categories and quivers\/}, math/0612206, 1--5

\bi [CL]{CL} V.\,Chari, S.\,Loktev, {\it Weyl, Fusion and Demazure
modules for the current algebra of $sl_{r+1\/}$}, Adv. Math. 207
(2006), no.\;2, 928--960

\bi [CP]{CP} V.\,Chari, A.\,Pressley
{\it Weyl Modules for Classical and Quantum Affine algebras\/},
Represent. Theory 5 (2001), 191--223 (electronic)

\bi [EG]{EG} A.\,Eremenko and A.\,Gabrielov,
{\it Rational functions with real critical points and the B.\,and
M.\,Shapiro conjecture in real enumerative geometry\/}, Ann. Math. (2)
{\bf 155} (2002), no.\;1, 105--129

\bi [EH] {EH} D.\,Eisenbud, J.\,Harris, {\it Limit Linear Series:
Basic Theory\/}, Invent. Mathem. {\bf 85}, (1986), 337--371

\bi [F]{F} E.\,Frenkel {\it Affine algebras, Langlands duality and
Bethe ansatz\/}, XIth International Congress of Mathematical Physics
(Paris, 1994), Int. Press, Cambridge, MA (1995), 606--642

\bibitem[Fu]{Fu} W.\,Fulton, {\it Intersection Theory\/}, Springer-Verlag, 1984

\bi [G] {G} M. Gaudin, La fonction d'onde de Bethe, Masson, Paris, 1983

\bi [GH]{GH} Ph.\,Griffiths, J.\,Harris, Principles of Algebraic Geometry,
Wiley, 1994

\bi[HP]{HP} W.\,Hodge and D.\,Pedoe, {\it Methods of Algebraic Geometry\/},
vol.\,2, Cambridge University Press, Cambridge, 1953

\bi [HU]{HU}
R.\,Howe, T.\,Umeda, {\it The Capelli identity, the double commutant theorem,
and multiplicity-free actions\/}, Math. Ann. {\bf 290} (1991), no.\;3, 565--619

\bi [K] {K} R.\,Kedem, {\it Fusion products of $\mathfrak{sl\/}_N$ symmetric
power representations and Kostka polynomials}, Quantum theory and
symmetries, World Sci. Publ., Hackensack, NJ, (2004), 88--93

\bi[M]{M} I.\,G.\,Macdonald, Symmetric functions and Hall Polynomials,
Oxford University Press, 1995

\bi [MNO]{MNO}
A.\,Molev, M.\,Nazarov, G.\,Olshanski,
{\it Yangians and classical Lie algebras\/}, Russian Math. Surveys
{\bf 51} (1996), no.\;2, 205--282

\bi [MTV1]{MTV1}
E.\,Mukhin, V.\,Tarasov, A.\,Varchenko,
{\it Bethe Eigenvectors of Higher Transfer Matrices\/},
J.~Stat. Mech. (2006), no.\;8, P08002, 1--44

\bi [MTV2]{MTV2}
E.\,Mukhin, V.\,Tarasov, A.\,Varchenko,
{\it The B.\,and M.\,Shapiro conjecture in real algebraic geometry
and the Bethe ansatz\/}, Preprint (2005), 1--18; {\tt math.AG/0512299}

\bi [MTV3]{MTV3} E.\,Mukhin, V.\,Tarasov, A.\,Varchenko,
{\it A generalization of the Capelli identity\/}, math/0610799 (2006), 1-14

\bi [MTV4]{MTV4} E.\,Mukhin, V.\,Tarasov, A.\,Varchenko, {\it Bethe
algebra and algebra of functions on the space of differential
operators of order two with polynomial solutions\/}, arXiv:0705.4114
(2007), 1-24

\bi [MV1]{MV1} E.\,Mukhin and A.\,Varchenko,
{\it Critical Points of Master Functions and Flag Varieties\/},
Communications in Contemporary Mathematics {\bf 6} (2004), no.\;1, 111--163

\bi [MV2]{MV2} E.\,Mukhin, A.\,Varchenko,
{\it Norm of a Bethe vector and the Hessian of the master function\/},
Compos. Math. {\bf 141} (2005), no.\;4, 1012--1028

\bi [MV3]{MV3} E.\,Mukhin, A.\,Varchenko, {\it Multiple orthogonal
polynomials and a counterexample to Gaudin Bethe Ansatz Conjecture\/},
math/0501144 (2005), 1--36

\bi [S] {S} F.\,Sottile, {\it Rational curves on Grassmannians: systems
theory, reality, and transversality\/}, Advances in algebraic geometry
motivated by physics (Lowell, MA, 2000), Contemp. Math., {\bf 276},
Amer. Math. Soc., Providence, RI (2001), 9--42

\bi [T]{T} D.\,Talalaev,
{\it Quantization of the Gaudin System\/}, Preprint (2004), 1--19;\\
{\tt hep-th/0404153}

\end{thebibliography}
\end{document}